\documentclass[reqno,10pt,a4paper]{amsart}

\numberwithin{equation}{section}

\usepackage{amsmath}

\usepackage{esint} % permette di fare l'integrale tagliato

\usepackage{amsthm}

\usepackage{epsfig}

\usepackage{psfrag}

\usepackage{graphicx}

\usepackage{graphpap,latexsym,epsf}

\usepackage{color}

\usepackage{amssymb,mathrsfs,enumerate}

\usepackage{endnotes}

\usepackage{calligra}
\usepackage{a4wide}

\usepackage{accents}
\usepackage{bbm}
\usepackage{dsfont}

\definecolor{citegreen}{rgb}{0,0.6,0}

\definecolor{refred}{rgb}{0.8,0,0}

\usepackage{enumitem}

\usepackage{mathtools}
\mathtoolsset{showonlyrefs}

\usepackage{palatino}
\usepackage[colorlinks, citecolor=citegreen, linkcolor=refred]{hyperref}

\newcommand{\R}{\mathbb{R}}

\newcommand{\N}{\mathbb{N}}

\newcommand{\SSS}{\mathbb{S}}

\mathchardef\emptyset="001F

\newcommand{\xx}{\hspace{.07em}}

\definecolor{vgreen}{rgb}{0.1,0.5,0.2}

\definecolor{viola}{RGB}{85,26,139}

\theoremstyle{plain}

\newtheorem{theorem}{Theorem}[section]

\newtheorem{proposition}[theorem]{Proposition}
\newtheorem{lemma}[theorem]{Lemma}

\theoremstyle{definition}
\newtheorem{definition}[theorem]{Definition}

\theoremstyle{remark}
\newtheorem{remark}[theorem]{Remark}

\definecolor{byzantium}{rgb}{0.44, 0.16, 0.39}

\definecolor{amber}{rgb}{1.0, 0.75, 0.0}

\definecolor{darkmagenta}{rgb}{0.55, 0.0, 0.55}

\definecolor{fuzzywuzzy}{rgb}{0.8, 0.4, 0.4}

\definecolor{brown}{rgb}{0.2, 0.08, 0.08}

\definecolor{arancio}{rgb}{1.0, 0.13, 0.0}

\begin{document}
\title[Decomposition of Vector Fields, $X$--ADM Mass, Higher--Dimensional Mass--Charge Inequalities]{Weighted Decomposition of Vector Fields, $X$--ADM Mass, and Higher--Dimensional Mass--Charge Inequalities}
\date{\today}

\author[Francesca Oronzio]{Francesca Oronzio}
\address[Francesca Oronzio]{Scuola Superiore Meridionale, Napoli, Italy}
\email[F. Oronzio]{f.oronzio@ssmeridionale.it}

\begin{abstract}
We study the $X$--ADM mass on complete, one--ended asymptotically flat Riemannian manifolds without boundary in arbitrary dimension $n\geqslant 3$. Starting from a weighted gradient--divergence-free decomposition of the vector field $X$, we construct a conformal metric whose scalar curvature is governed by the critical modified scalar curvature $\mathrm{R}_{X}^{(1-n)}$ associated with $X$. This yields an exact decomposition of the $X$--ADM mass into the ADM mass of the conformal metric plus a nonnegative defect term, which vanishes precisely when $X$ is a gradient, thereby giving an alternative decomposition--based proof of positivity with a full rigidity statement. Explicit negative--mass examples show that the range of parameters $k$ in the modified scalar curvature $\mathrm{R}_{X}^{(k)}$ condition is sharp.\\
We then apply the $X$--positive mass theorem to systems of vector fields, obtaining higher--dimensional multi--charge mass inequalities with rigidity and global alignment in the equality case. This yields the classical three--dimensional electric--magnetic inequality and its higher--dimensional analogue. In higher dimensions, the magnetic datum $\beta$ is a $2$--form and does not have a canonical scalar charge on the asymptotically flat end. We prove that a selected tensorial flux vanishes when $d\beta \in L^1$. Under a prescribed critical asymptotic ansatz, this flux defines a non--trivial tensorial magnetic charge and we formulate a conditional reduction to the $X$--positive mass theorem.
\end{abstract}

\maketitle

\bigskip



\section{Introduction}

Throughout this paper, $(M, g)$ is a Riemannian manifold without boundary. We denote by $\mathrm{R}$ the scalar curvature, by $\nabla$ the Levi--Civita connection, and by $\Delta=\mathrm{tr}\nabla d$ the associated Laplace--Beltrami operator. Moreover, for any $\ell$--form $\omega$ on $M$, we denote by $\|\omega\|^2=g(\omega,\omega)/\ell!$ the squared normalized norm of $\omega$. Finally, we drop the explicit dependence on the metric for all quantities computed with respect to $g$ and we use the Einstein summation convention.

In order to state precisely our main result, we recall the definition of an asymptotically flat manifold and introduce the notion of $X$--ADM mass.

\begin{definition}\label{defAFman}
An $n$--dimensional Riemannian manifold $(M,g)$, with $n\geqslant 3$, is said to be {\em asymptotically flat} if there exists a compact subset $K$ such that $U_\infty=M\setminus K$ is diffeomorphic to $\R^{n}$ minus a closed ball $\overline{B}_{r}(0)$ by means of a coordinate map $\varphi: U_\infty\to\R^n\setminus\overline{B}_{r}(0)$ such that
\begin{equation}\label{eq1}
g_{ij}=\delta_{ij} + O_{2}(\vert x \vert ^{-\tau}),
\end{equation}
for some $\tau>(n-2)/2$ and all $i,j\in \{1,\dots,n\}$. The pair $(U_\infty, \varphi)$ will be called an {\em asymptotically flat coordinate chart} of $(M,g)$ and the real number $\tau$ will be called the {\em order of decay} (briefly, {\em the order}) of the metric $g$ in such an asymptotically flat coordinate chart.
\end{definition}

Above, we adopted the Landau big--$O$ convention (which can be easily adapted to manifolds that are diffeomorphic, outside a compact subset, to $\R^{n}$ minus a closed ball with center at the origin). We briefly recall it. Let $f$ be a smooth real--valued function defined outside a compact set of $\R^n$, and let $\tau\in \R$. We write $f=O_k(|\xx x\xx |^{-\tau})$ if there exists a constant $C>0$ such that $\vert \xx \partial^{\alpha}f(x)\xx\vert \leqslant C\vert \xx x\xx\vert ^{-|\xx \alpha\xx|-\tau}$ for every $x\in \R^n\setminus B_R(0)$, when $R>0$ is sufficiently large, for every multi--index $\alpha$ with $0 \leqslant |\xx \alpha\xx| \leqslant k$.

\begin{definition}[$X$--ADM mass]\label{defADMmass}
Let $(M^{n},g)$ be an asymptotically flat manifold. A smooth vector field $X$ on $M$ is said to be {\em admissible} if $X\in C^1_{-1-\tau_0}(M;TM)$ for some $\tau_0>(n-2)/2$.\\
Given an admissible smooth vector field $X$, we assume that $\mathrm{R}+2\,\mathrm{div}(X)\in L^1(M,g)$. Under this assumption, the following limit computed with respect to a given asymptotically flat chart exists and is finite:
\begin{equation}\label{formXADMmass}
m_{X}(g)=\frac{1}{2(n-1)|\SSS^{n-1}|}\lim_{r\to +\infty}\int\limits_{\{\vert x \vert\,=\,r\}}\!\!\!\delta_{j}^i(\delta^{kl}\partial_{k}g_{il}-\delta^{kl}\partial_{i}g_{kl}+2\delta_{ik}X^k)\frac{x^{j}}{\vert x \vert}\,d\sigma_{\mathrm{eucl}}.
\end{equation}
We call it {\em $X$--ADM mass} of $(M,g)$. We remark that, under the assumption that $\mathrm{R}+2\,\mathrm{div}(X)$ is integrable, the existence, finiteness, and invariance under changes of asymptotically flat coordinates of the above limit follow by a direct adaptation of the classical arguments of Bartnik~\cite{Bartnik} and Chru\'sciel~\cite{Chrusciel}. Indeed, the additional boundary flux $2X$ corresponds, through the divergence theorem, to replacing the scalar--curvature term $\mathrm{R}$ by $\mathrm{R}+2\,\mathrm{div}(X)$, while all remaining asymptotic error terms are integrable under the stated decay assumptions.
\end{definition}

A principal aim of this paper is to present a new decomposition--based proof of the $n$--dimensional positive mass theorem, with $n\geqslant 3$, for the $X$--ADM mass.
 
\begin{theorem}[$X$--positive mass theorem with a new decomposition]\label{GenPMT}
Let $(M^n,g)$, $n\geqslant 3$, be a connected, complete, asymptotically flat manifold of order $\tau>(n-2)/2$, and let $X$ be an admissible smooth vector field on $M$. Assume that:
\begin{itemize}
\item[$(a)$] $\mathrm{R}+2\,\mathrm{div}(X)\in L^{1}(M,g)$;
\item[$(b)$] $X\in C^{1,\alpha}_{-1-\delta}(M;TM)$ for some $\alpha\in (0,1)$ and $\delta\in \big((n-2)/2, n-2\big)$;
\item[$(c)$] $\mathrm{R}_X^{(k)}=\mathrm{R}+2\,\mathrm{div}(X)-(1+1/k)\, \vert X\vert^2\, \geqslant\,0$ on $M$, for some $k\in \R\setminus (1-n,0]$;
\end{itemize}
Let $X=\nabla f+Y$ be the unique weighted gradient–divergence-free decomposition provided by Proposition \ref{thm:main-reduction}, where $f\in C^{2,\alpha}_{-\delta}(M)$ is a smooth function and $Y\in C^{1,\alpha}_{-1-\delta}(M;TM)$ is a divergence--free smooth vector field such that $\nabla f $ and $Y$ are orthogonal with respect to the $L^2$--inner product. Then, we have
\begin{equation}\label{feqreductionX-masstogradientmasswithoutboundary}
m_X(g)=m_{\nabla f}(g).
\end{equation}
Define the operator $L$ by
$$L\phi=\Delta \phi-\frac{n-2}{n-1}\,g\left(\nabla f,\nabla \phi\right)-\frac{1}{2}\left(\frac{n-2}{n-1}\right)^{\!2}\left(g(\nabla f,Y)+\frac{1}{2\,}\vert Y\vert^2\right)\phi $$
for any $\phi\in C^2(M)$. Let $u$ be the unique smooth positive function satisfying
\begin{equation}\label{feq34}
  Lu=0\quad\text{in }\,\,M,
\qquad
  u=1+v \quad \text{with}\,\,v\in C^{2,\alpha}_{-\min\{\delta,\tau\}}(M),
\end{equation}
so that the complete asymptotically flat metric 
$$\widehat g=u^{\frac{4}{n-2}} e^{-\frac{2f}{n-1}} g$$ has
$$\mathrm{R}_{\widehat g}=u^{-\frac{4}{n-2}} e^{\frac{2f}{n-1}} \mathrm{R}_X^{(1-n)}=u^{-\frac{4}{n-2}} e^{\frac{2f}{n-1}}\bigg[\mathrm{R}_X^{(k)}
+
\left(
\frac{1}{k}+\frac{1}{n-1}
\right)|X|^2\bigg] \geq 0.$$
Then, we have 
\begin{equation}\label{GenPMI-main}
m_{X}(g)=m_{\mathrm{ADM}}\big(\widehat g \big)+ \frac{2}{(n-2)|\mathbb{S}^{n-1}|}\,\int\limits_M e^{-(\frac{n-2}{n-1})f}\left|\nabla u+\frac{n-2}{2(n-1)} uY\right|^2\,d\mu \geq  0.
\end{equation}
Furthermore, the following statements are true.
\begin{enumerate}
\item If $\mathrm{R}_X^{(k)} \geq 0$ with $k\in \R\setminus [1-n,0]$, $m_{X}(g)=0$ if and only if $X$ vanishes on $M$ and $(M,g)$ is isometric to $(\R^n,g_{\mathrm{eucl}})$.
\item If $\mathrm{R}_X^{(1-n)}\geq 0$, $m_{X}(g)=0$ if and only if $X=\nabla f$ and $(M,e^{-2f/(n-1)}g)$ is isometric to $(\R^n,g_{\mathrm{eucl}})$.
\end{enumerate}
\end{theorem}

The negative--mass examples constructed in Section~\ref{Sectsharpinterval} show that the
excluded parameter interval is sharp.

The theorem above belongs to a line of developments extending the Riemannian positive mass theorem beyond the unweighted scalar--curvature setting. The classical positive mass theorem was proved by Schoen and Yau \cite{SchYau79} in dimensions $3\leqslant n\leqslant 7$ and by Witten \cite{Witten} for spin manifolds. Recent work has provided higher--dimensional formulations \cite{BHHSZ2026, CMS, CMSW, Loc, ScY2022}, culminating in all-dimensional formulations in both the smooth category and in settings allowing for controlled singularities \cite{BrendleWang2026, KhuriWangWang}. In the weighted setting, Baldauf and Ozuch \cite{BaOz2002} introduced the weighted ADM mass and established a spinorial proof of the positive mass theorem. Chu and Zhu \cite{ChuZhu} subsequently obtained a non--spin proof, while Law, Lopez, and Santiago \cite{LaLoSa} gave a unified treatment for smooth metric measure spaces and showed that the weighted positive mass theorem is equivalent, through a conformal change, to the usual one.

The $X$--ADM mass was introduced by the author and Mantegazza \cite{MaOr} for asymptotically flat three--manifolds as a vector field extension of the weighted mass. In \cite{MaOr}, positivity and rigidity were established by means of a monotonicity formula along the level sets of the Green function for an appropriate drift operator, under an additional topological assumption. This framework includes the weighted setting when $X$ is a gradient and, in particular, recovers the classical mass--charge inequality. The physical origins of this inequality trace back to Gibbons and Hull \cite{GibbHull1982}, while its first proof was given by Gibbons, Hawking, Horowitz, and Perry \cite{GHHP1983}.
Subsequent developments provided several extensions and alternative approaches to the charged positive mass theorem. Herzlich \cite{Her98} and Bartnik--Chruściel \cite{BC05} provided rigorous analytic treatments in the presence of an inner boundary. Further three--dimensional results were obtained by Jaracz \cite{Jar20}, who treated the time--symmetric case via inverse mean curvature flow and also allowed for asymptotically cylindrical ends, and by Bray, Hirsch, Kazaras, Khuri, and Zhang \cite{BHKKZ23}, who developed a spacetime harmonic function approach covering both asymptotically flat and cylindrical ends. Charged Penrose--type inequalities were also established for multiple black holes by Khuri, Weinstein, and Yamada \cite{KWY17}, and in higher dimensions by de Lima, Gir$\tilde{\text{a}}$o, Loz\'{o}rio, and Silva \cite{dLGaLS16}. Raulot \cite{Rau25} established a higher--dimensional spinorial positive energy theorem for charged initial data containing at least one asymptotically flat end. Finally, the equality case in the charged positive energy theorem has been studied in \cite{Tod83, CRT06, KW13, BHKKZ23}.

More recently, McCormick \cite{McC} established the $X$--positive mass
theorem in arbitrary dimension by means of a scalar--flat conformal
reduction. This approach relates the result directly to the classical
positive mass theorem and also yields mass--charge and boundary
versions.

Theorem \ref{GenPMT} follows a complementary, decomposition--based route. The
weighted gradient--divergence-free splitting separates the gradient
component, which determines the mass at infinity, from the
divergence--free component; a further conformal correction then removes
the latter. The resulting formula contains an exact nonnegative defect
term, which vanishes precisely when $X$ is a gradient. In the gradient
case, the construction therefore recovers the weighted conformal
framework, whereas for general vector fields it provides an alternative
decomposition--based approach to McCormick's scalar--flat reduction.

More precisely, consider the decomposition $X=\nabla f+Y$ with $\mathrm{div}(Y)=0$.
Since $Y$ has vanishing flux at infinity, the
$X$--ADM mass depends only on the gradient component of $X$, namely $m_X(g)=m_{\nabla f}(g)$.
On the other hand, Law, Lopez, and Santiago
\cite[Lemma~2.5]{LaLoSa} proved that, with the corresponding
normalization of the mass,
$$
m_{\nabla f}(g)=m_{\mathrm{ADM}}(g_f),
$$
where $g_f$ is the complete asymptotically flat metric defined as $g_f=e^{-\frac{2f}{n-1}}g$.
Consequently, the mass identity \eqref{GenPMI-main} can equivalently be written
as
$$
m_{\mathrm{ADM}}(g_f)
=
m_{\mathrm{ADM}}(\widehat g)+2\mathcal D \quad \text{with}\quad \mathcal D= \frac{1}{(n-2)|\mathbb{S}^{n-1}|}\,\int\limits_M e^{-(\frac{n-2}{n-1})f}\left|\nabla u+\frac{n-2}{2(n-1)} uY\right|^2\,d\mu.
$$
Thus, $2\mathcal D$ is the exact variation of the ADM mass produced by the
additional conformal factor $u^{4/(n-2)}$. This factor arises from the
correction $w=-[2(n-1)/(n-2)]\log u$, which replaces the divergence--free
component $Y$ by a gradient while preserving the critical modified
scalar curvature:
$$
\mathrm{R}_{\nabla(f+w)}^{(1-n)}
=
\mathrm{R}_X^{(1-n)}.
$$
In particular, $\mathcal D=0$ precisely when $Y\equiv0$. In this case
$u\equiv1$, $w\equiv0$, and $\widehat g=g_f$, so the identity above
reduces exactly to the conformal mass identity of
Law--Lopez--Santiago.

For comparison, McCormick \cite[Remark~1.1]{McC}
chooses a positive function $\varphi$ so that
$$
\widetilde g=\varphi^{\frac{4}{n-2}}g
$$
is scalar--flat and obtains, after adjusting for the normalization of the
mass used here,
$$
m_X(g)
=
m_{\mathrm{ADM}}(\widetilde g)
+
\frac{1}{2(n-1)|\mathbb S^{n-1}|}
\int_M
\left[
c_n
\left|
\nabla\varphi+\frac{2}{c_n}\varphi X
\right|^2
+
R_X^{(1-n)}\varphi^2
\right]
\,d\mu_g,
$$
where $c_n$ denotes the constant $4(n-1)/(n-2)$.
Thus, McCormick's remainder is a bulk term arising from a scalar--flat
conformal reduction applied directly to the full vector field $X$.
By contrast, the metric $\widehat g$ considered here is not required
to be scalar--flat; rather,
$$
\mathrm{R}_{\widehat g}
=
e^{\frac{2(f+w)}{n-1}}\mathrm{R}_X^{(1-n)},
$$
and the exact mass variation associated with the elimination of the
divergence--free component is encoded by the single coefficient
$-\mathcal D$ in the asymptotic of $u$. Hence the present identity may be viewed as a refinement of the
gradient conformal identity of Law--Lopez--Santiago and as a different,
decomposition--based counterpart of McCormick's scalar--flat reduction.
\smallskip

Although the $X$--positive mass theorem is formulated for a single vector field, it naturally applies to systems carrying several vector fields, thereby obtaining the following $N$--charge positive mass theorem.

\begin{theorem}[$N$--charge positive mass theorem]\label{thm:Nchargepostivemass}
Let $(M,g)$ be a connected, complete, asymptotically flat manifold, of dimension $n\geq 3$. Let $N\in\N_{>0}$. Let $\mathcal{E}_1, \dots, \mathcal{E}_N$ be admissible smooth vector fields on $M$. Assume that:
\begin{itemize}
\item[$(a)$] $\mathcal{E}_1, \dots, \mathcal{E}_N\in C^{1,\alpha}_{-1-\delta}(M;TM)$ for some $\alpha\in (0,1)$ and $\delta\in \big((n-2)/2, n-2\big)$;
 \item[$(b)$] $\mathrm{R},\ \mathrm{div}\mathcal{E}_1,\ \dots,\ \mathrm{div}\mathcal{E}_N \in L^1(M,g)$;
 \item [$(c)$] Set
 \begin{align}
  2(n-1)\mu&=\mathrm{R}-(n-1)(n-2)\bigl(|\mathcal{E}_1|^2+\dots+| \mathcal{E}_N|^2\bigr),\label{eq:generalized-matter-densityintro}\\
\rho_{\mathcal{E}_i}&=\mathrm{div} \mathcal{E}_i \quad \text{for every $i\in \{1,\dots, N\}$}, \label{eq:generalized-charge-densitiesintro}
\end{align}
the $N$--charge dominant energy condition
\begin{equation}
\mu\geq\sqrt{\rho_{\mathcal{E}_1}^2+\dots+\rho_{\mathcal{E}_N}^2}
\label{eq:n-dimensional-N-charge-decintro}
\end{equation}
holds on $M$, that is,
\begin{equation}
 \mathrm{R}-(n-1)(n-2)\bigl(|\mathcal{E}_1|^2+\dots+| \mathcal{E}_N|^2\bigr)
 \geq2(n-1)\sqrt{(\mathrm{div} \mathcal{E}_1)^2+\dots+(\mathrm{div} \mathcal{E}_N)^2}
 \label{eq:n-dimensional-N-charge-dec-expandedintro}
\end{equation}
on $M$.
\end{itemize}
Define $$\mathcal{Q}=\sqrt{\mathcal{Q}_{\mathcal{E}_1}^2+\dots+\mathcal{Q}_{\mathcal{E}_N}^2}.$$
Then, we have
\begin{equation}
 m_{\mathrm{ADM}}(g)\geq \mathcal{Q},
 \label{eq:n-dimensional-N-charge-pmtintro}
\end{equation}
where the charges $\mathcal{Q}_{\mathcal{E}_\ell}$ are given by
\begin{equation}
 \mathcal{Q}_{\mathcal{E}_\ell}=\frac{1}{ |\SSS^{n-1}|}
 \lim_{r\to+\infty}\int\limits_{\{\vert x \vert\,=\,r\}}\!\!\!\delta_{ij}\mathcal{E}_\ell^i \frac{x^{j}}{\vert x \vert}\,d\sigma_{\mathrm{eucl}},\quad \text{for every $\ell\in \{1,\dots, N\}$}.
 \label{eq:n-dimensional-N-chargesintro}
\end{equation}
Furthermore, the following statements are true.
\begin{enumerate}
\item If $\mathcal{Q}=0$, equality holds in \eqref{eq:n-dimensional-N-charge-pmtintro} if and only if $\mathcal{E}_i\equiv 0$ for every $i\in \{1,\dots, N\}$ and $(M,g)$ is isometric to $(\R^n, g_{\mathrm{eucl}})$.
\item If $\mathcal{Q}>0$, equality holds in \eqref{eq:n-dimensional-N-charge-pmtintro} if and only if there exists a smooth function $f\in C^{2,\alpha}_{-\delta}(M)$ such that 
\begin{align}
& \nabla f=-(n-1)\,\left( \frac{\mathcal{Q}_{\mathcal{E}_1}}{\mathcal{Q}}\, \mathcal{E}_1+\dots+\frac{\mathcal{Q}_{\mathcal{E}_N}}{\mathcal{Q}}\,\mathcal{E}_N \right),\label{gradformofX}\\
 \Big(M,\,&e^{-2f/(n-1)}g\Big)\,\text{is isometric to}\,
 (\mathbb R^n, g_{\mathrm{eucl}}).
 \label{eq:critical-rigidity-conformal-flat-nintro}
\end{align}
Moreover, whenever equality holds, we have
\begin{align}
& \,\,\mathcal{E}_{i}=\frac{\mathcal{Q}_{\mathcal{E}_i}}{\mathcal{Q}}\,\left(\frac{\mathcal{Q}_{\mathcal{E}_1}}{\mathcal{Q}}\mathcal{E}_1+\dots+\frac{\mathcal{Q}_{\mathcal{E}_N}}{\mathcal{Q}}\mathcal{E}_N  \right), \label{linkEieZb}\\
&\,\, \rho_{\mathcal{E}_i}=\frac{\mathcal{Q}_{\mathcal{E}_i}}{\mathcal{Q}}\,\mu.
\end{align}
for every $i\in \{1,\dots, N\}$.
\item If $\mathrm{div} \mathcal{E}_i\equiv 0$ for all $i\in \{1,\dots, N\}$, equality in
\eqref{eq:n-dimensional-N-charge-pmtintro} occurs if and only if $\mathcal{E}_i\equiv 0$ for every $i\in \{1,\dots, N\}$ and $(M,g)$ is isometric to $(\R^n,g_{\mathrm{eucl}})$.
\end{enumerate}
\end{theorem}

In the one--ended boundaryless setting considered here, a globally divergence--free smooth field has zero total charge. Thus, nonzero charges necessarily correspond to nontrivial source densities and source--free charged black--hole data require a different global setting, such as an inner boundary or an additional end.

The preceding result has a natural interpretation in the three--dimensional Einstein--Maxwell setting. Let $(M^{3},g)$ be a time--symmetric initial data set and let $\mathcal E$ and $\mathcal B$ denote its electric and magnetic fields. The magnetic part of the Maxwell two--form can be identified, through the Hodge operator, with a one--form and hence with a vector field. The two fields therefore have the same tensorial type and fit into the $N$--field framework with $N=2$. Their divergences determine the electric and magnetic source densities, while their cross product gives the electromagnetic momentum density appearing in the Einstein--Maxwell constraint equations. After subtracting the electromagnetic contribution from the scalar curvature, the charged dominant energy condition requires the remaining energy density to control simultaneously the electric and magnetic source densities and the electromagnetic momentum density. In particular, it implies that the two--charge energy condition used in the preceding theorem holds. Therefore, applying the $N$--charge theorem to the pair $(\mathcal E, \mathcal B)$ yields the following three--dimensional charged positive mass theorem.

\begin{theorem}[$3$D charged positive mass theorem]\label{crl:time-symmetric-dyonic-pmt}
Let $(M^3,g)$ be a connected, oriented, complete, asymptotically flat Riemannian manifold. Let $\mathcal{E}$ and $\mathcal{B}$ be admissible smooth vector fields on $M$. Assume that:
\begin{itemize}
\item[$(a)$] $\mathcal{E}, \mathcal{B}\in C^{1,\alpha}_{-1-\delta}(M;TM)$ for some $\alpha\in (0,1)$ and $\delta\in \big(1/2, 1\big)$;
 \item[$(b)$] $\mathrm{R},\ \mathrm{div}\mathcal{E}, \mathrm{div}\mathcal{B} \in L^1(M,g)$;
 \item [$(c)$] Set
 \begin{equation}
\quad\quad 4\mu
=\mathrm{R}-2|\mathcal{E}|^2-2|\mathcal{B}|^2,
\qquad
\rho_{\mathcal{E}}=\mathrm{div} \mathcal{E},
\qquad
\rho_{\mathcal{B}}=\mathrm{div}\mathcal{B},
\qquad
J_{\mathrm{em}}=\star (\mathcal{E}^\flat\wedge \mathcal{B}^\flat),
\label{eq:time-symmetric-densitiesintro}
\end{equation}
the charged dominant energy condition 
\begin{equation}
\mu\geq \sqrt{\rho_\mathcal{E}^2+\rho_\mathcal{B}^2+|J_{\mathrm{em}}|^2}
\label{eq:charged-dec}
\end{equation}
holds on $M$. Here, $\star:\omega\in \Omega^2M\to\Omega^1M $ is the Hodge star operator.
\end{itemize}
Then,
\begin{equation}
m_{\mathrm{ADM}}(g)\geq \sqrt{\mathcal{Q}_{\mathcal{E}}^2+\mathcal{Q}_{\mathcal{B}}^2}
\label{eq:dyonic-mass-chargeintro}
\end{equation}
where $\mathcal{Q}_{\mathcal{E}}$ and $\mathcal{Q}_{\mathcal{B}}$ are the electric and magnetic charges, defined respectively as
\begin{equation}
\mathcal{Q}_{\mathcal{E}}=\frac{1}{4\pi} \lim_{r\to+\infty}\int\limits_{\{\vert x \vert\,=\,r\}}\!\!\!\delta_{ij}\mathcal{E}^i \frac{x^{j}}{\vert x \vert}\,d\sigma_{\mathrm{eucl}},
\qquad
\mathcal{Q}_{\mathcal{B}}=\frac{1}{4\pi} \lim_{r\to+\infty}\int\limits_{\{\vert x \vert\,=\,r\}}\!\!\!\delta_{ij}\mathcal{B}^i \frac{x^{j}}{\vert x \vert}\,d\sigma_{\mathrm{eucl}}.
\label{eq:electric-magnetic-chargesintro}
\end{equation}
Furthermore, the following statements are true.
\begin{enumerate}
\item If $\mathcal{Q}_{\mathcal{E}}=0=\mathcal{Q}_{\mathcal{B}}$, equality holds in \eqref{eq:dyonic-mass-chargeintro} if and only if $\mathcal{E}\equiv 0$, $\mathcal{B}\equiv 0$ and $(M,g)$ is isometric to $(\R^3, g_{\mathrm{eucl}})$.
\item Let $\mathcal{Q}=\sqrt{\mathcal{Q}_{\mathcal{E}}^2+\mathcal{Q}_{\mathcal{B}}^2}$. If $\mathcal{Q}>0$, equality holds in \eqref{eq:dyonic-mass-chargeintro} if and only if there exists a smooth function $f\in C^{2,\alpha}_{-\delta}(M)$ such that 
\begin{align}
& \nabla f=-2\,\left( \frac{\mathcal{Q}_{\mathcal{E}}}{\mathcal{Q}}\, \mathcal{E}+\frac{\mathcal{Q}_{\mathcal{B}}}{\mathcal{Q}}\,\mathcal{B} \right),\label{gradformofXelmag}\\
 \Big(M,\,&e^{-f}g\Big)\,\text{is isometric to}\,
 (\mathbb R^3, g_{\mathrm{eucl}}).
 \label{eq:critical-rigidity-conformal-flat-nintroelmag}
\end{align}
Moreover, whenever equality holds, we have
\begin{align}
&\,\,J_{\mathrm{em}}=0\\
& \,\,\mathcal{E}=\frac{\mathcal{Q}_{\mathcal{E}}}{\mathcal{Q}}\,\left( \frac{\mathcal{Q}_{\mathcal{E}}}{\mathcal{Q}}\, \mathcal{E}+\frac{\mathcal{Q}_{\mathcal{B}}}{\mathcal{Q}}\,\mathcal{B} \right), \label{linkEeZb}\\
& \,\,\mathcal{B}=\frac{\mathcal{Q}_{\mathcal{B}}}{\mathcal{Q}}\,\left( \frac{\mathcal{Q}_{\mathcal{E}}}{\mathcal{Q}}\, \mathcal{E}+\frac{\mathcal{Q}_{\mathcal{B}}}{\mathcal{Q}}\,\mathcal{B} \right), \label{linkBeZb}\\
&\,\, \rho_{\mathcal{E}}=\frac{\mathcal{Q}_{\mathcal{E}}}{\mathcal{Q}}\,\mu,\quad \text{and}\quad \rho_{\mathcal{B}}=\frac{\mathcal{Q}_{\mathcal{B}}}{\mathcal{Q}}\,\mu.
\end{align}
\item If $\mathrm{div} \mathcal{E}\equiv 0$ and $\mathrm{div} \mathcal{B}\equiv 0$, equality in
\eqref{eq:dyonic-mass-chargeintro} occurs if and only if $\mathcal{E}\equiv 0$, $\mathcal{B}\equiv 0$ and $(M,g)$ is isometric to $(\R^3, g_{\mathrm{eucl}})$.
\end{enumerate}
\end{theorem}

The three--dimensional charged theorem relies crucially on the fact that the electric and magnetic fields share the same tensorial type. In higher dimensions, this property is lost because, in the standard Maxwell theory, the electromagnetic field is represented by a 2--form on the given spacetime. Its decomposition along a spacelike hypersurface yields an electric one--form and a magnetic two--form, or equivalently an $(n-2)$--form after applying the Hodge star operator. While their interaction determines the electromagnetic momentum one--form, the mismatch in tensorial degrees prevents the constant electric--magnetic rotations that are essential to the preceding two--charge argument. Moreover, on an asymptotically flat end, the electric flux gives rise to a canonical scalar charge, whereas the magnetic two--form does not, in general, define a second scalar charge. Preserving this Maxwell decomposition therefore leads to the following higher--dimensional result.

\begin{theorem}[Maxwell--form charged positive mass theorem]\label{thm:canonical-Maxwell-PMTintro}
Let $(M^n,g)$, $n\geq 3$, be a connected, oriented, complete, asymptotically flat Riemannian manifold. Let $\mathcal E$ be a smooth admissible vector field satisfying
$$
\mathcal E\in C^{1,\alpha}_{-1-\delta}(M;TM),
\qquad
\frac{n-2}{2}<\delta<n-2,
$$
for some $\alpha\in(0,1)$, and assume that
$$
\mathrm{R},\ \mathrm{div} \mathcal E\in L^1(M,g).
$$
We denote by $e\in\Omega^1(M)$ the smooth electric one--form $\mathcal E^\flat$, and let $\beta\in\Omega^2(M)$ be a smooth magnetic two--form. We then define the magnetic $(n-2)$--form and the electromagnetic momentum one--form as
$$
b=\star\beta\in\Omega^{n-2}(M),
\qquad
J_{\mathrm{em}}
=-\iota_{\mathcal E}\beta
=(-1)^{n+1}\star(e\wedge b),
$$
respectively. Moreover, we set $\rho_{\mathcal E}=\mathrm{div} \mathcal E$ and
$$\mu=\frac{1}{2(n-1)}
\left[
\mathrm{R}-(n-1)(n-2)
\bigl(|\mathcal E|^2+\|b\|^2\bigr)
\right].
$$
Here, $\|\omega\|^2=g(\omega,\omega)/\ell!$ denotes the squared (normalized) norm of a $\ell$--form $\omega$.
Then, if
\begin{equation}\label{feq19}
  \mu
  \geq
  \sqrt{\rho_{\mathcal E}^2+(n-2)^2\,|J_{\mathrm{em}}|^2}
  \quad\text{on }M,
\end{equation}
we have
\begin{equation}\label{feq20}
  m_{\mathrm{ADM}}(g)\geq |\mathcal{Q}_{\mathcal{E}}|,
\end{equation}
where $\mathcal{Q}_{\mathcal{E}}$ is the electric charge given by
\begin{equation}
\mathcal{Q}_{\mathcal{E}}=\frac{1}{|\SSS^{n-1}|} \lim_{r\to+\infty}\int\limits_{\{\vert x \vert\,=\,r\}}\!\!\!\delta_{ij}\mathcal{E}^i \frac{x^{j}}{\vert x \vert}\,d\sigma_{\mathrm{eucl}}.
\end{equation}
Furthermore, the following statements are true.
\begin{enumerate}
\item If $\mathcal{Q}_{\mathcal{E}}=0$, equality holds if and only if $\mathcal E\equiv0$, $b\equiv0$ and $(M,g)$ is isometric to $(\R^n, g_{\mathrm{eucl}})$.
\item If $\mathcal{Q}_{\mathcal{E}}\neq 0$, equality holds if and only if there exists a smooth function $f\in C^{2,\alpha}_{-\delta}(M)$ such that 
\begin{align}
& \nabla f=-(n-1)\,\frac{\mathcal{Q}_{\mathcal{E}}}{|\mathcal{Q}_{\mathcal{E}}|}\, \mathcal{E}\label{feq21}\\
 \Big(M,\,&e^{-\frac{2f}{n-1}}g\Big)\,\text{is isometric to}\,
 (\mathbb R^n, g_{\mathrm{eucl}}).\label{feq22}
\end{align}
Moreover, whenever equality holds, the magnetic field and, consequently, the electromagnetic momentum one--form vanish identically and the energy condition is saturated:
\begin{equation}\label{feq23}
b\equiv0,
\qquad
\mu=|\rho_{\mathcal{E}}|.
\end{equation}
\end{enumerate}
\end{theorem}

The above theorem shows that the magnetic field contributes as a nonnegative term to the energy condition, while the resulting mass bound involves only the electric scalar charge. This is not merely a limitation of the charge--selection argument. One may attempt to associate with the magnetic $2$--form a tensor--valued charge. Under the standard Coulomb decay and integrability assumptions, however, this tensor--valued charge vanishes.

\begin{proposition}[Vanishing of the tensorial magnetic charge in the $L^1$ regime]\label{prop:tensorial-magnetic-chargeintro}
Let $(M^n,g)$, $n>3$, be a connected, oriented, complete, asymptotically flat Riemannian manifold. Consider a smooth magnetic two--form $\beta\in\Omega^2(M)$ and set $b=\star\beta\in\Omega^{n-2}(M)$. We assume that 
\begin{equation}
\beta\in C_{1-n}^0\bigl(M;\Lambda^2T^*M\bigr) \quad\quad \text{and}\quad\quad
   d\beta\in L^1\bigl(M;\Lambda^3T^*M\bigr).
   \label{eq:beta-coulomb-decay and dbeta-L1}
\end{equation}
Then, in a fixed positively oriented asymptotically flat coordinate chart $(U_\infty, \varphi=(x^1,\dots, x^n))$ of $(M,g)$, the limits
\begin{equation} \label{eq:directional-magnetic-chargebis}
  (\mathbb Q_b)_{i_1\dots i_{n-3}}
   =\frac{(-1)^{n-1}}{|\SSS^{n-1}|}
     \lim\limits_{r\to +\infty}\int\limits_{\{\vert x \vert\,=\,r\}}\!\!\! b_{i_1\dots i_{n-3}j}\frac{x^{j}}{\vert x \vert}\,d\sigma_{\mathrm{eucl}}
\end{equation}
exist and are finite. Thus, it is well defined 
\begin{equation} \label{eq:directional-magnetic-charge}
\mathbb Q_b=\!\!\!\sum_{1\leqslant i_1<\dots<i_{n-3}\leqslant n}\!\!\!(\mathbb Q_b)_{i_1\dots i_{n-3}}e^{i_1}\wedge \dots\wedge e^{i_{n-3}}\in \Lambda^{n-3}(\mathbb R^n)^*,
\end{equation}
where $\{e^1,\dots, e^n\}$ is the dual basis associated with the canonical basis $\{e_1,\dots, e_{n}\}$ of $\R^n$. Under the integrability assumption \eqref{eq:beta-coulomb-decay and dbeta-L1},
\begin{equation}
   \mathbb Q_b=0.
   \label{eq:QB-L1-vanishingintro}
\end{equation}
In particular, the limits remain unchanged if the coordinate spheres are replaced by any homologous family of hypersurfaces exhausting the end, and are independent of the fixed positively oriented asymptotically flat chart used to compute them.
\end{proposition}

Consequently, within the asymptotically flat $L^{1}$--framework, the tensorial flux associated with the magnetic two--form vanishes at infinity. Indeed, the assumption $d\beta\in L^1$, together with Stokes’ theorem and a coarea argument, forces the limiting flux to be zero. Accordingly, the selected tensorial magnetic charge vanishes, and the resulting mass bound involves only the electric charge. This conclusion concerns only the selected tensorial flux and does not imply that the magnetic two--form $\beta$, or its Hodge dual $b$, vanishes identically.

To obtain a nontrivial tensorial magnetic invariant, one must therefore leave the $L^1$--regime for $d\beta$. We accordingly consider a critical asymptotic behavior of the magnetic $(n-2)$--form $b$, for which $d\beta$ may exhibit the borderline $r^{-n}$ decay and need not be integrable. In this regime, the preceding vanishing result no longer applies, and a nonzero asymptotic flux may persist.

\begin{theorem}[Selected tensorial mass--charge inequality]\label{thm:selected-tensorial-mass-charge}
Let $(M^n,g)$, with $n>3$, be a connected, oriented, complete, asymptotically flat Riemannian manifold, and let $\mathcal E$ be a smooth admissible vector field. We assume that $\mathrm{div} \mathcal E\in L^1(M,g)$
and consider the electric charge
\begin{equation}
\mathcal{Q}_{\mathcal{E}}=\frac{1}{|\SSS^{n-1}|} \lim_{r\to+\infty}\int\limits_{\{\vert x \vert\,=\,r\}}\!\!\!\delta_{ij}\mathcal{E}^i \frac{x^{j}}{\vert x \vert}\,d\sigma_{\mathrm{eucl}}.
\end{equation}
Let $\beta\in\Omega^2(M)$ be a smooth magnetic two--form and set $b=\star\beta\in\Omega^{n-2}(M)$.
We suppose that there exists a distinguished positively oriented asymptotically flat chart $(U_{\infty}, \varphi=(x^1,\dots,x^n))$ of order $\tau>(n-2)/2$ such that
\begin{align}
b&=\frac{n}{3}\,r^{1-n}\,\bigl(dr\wedge\mathbb Q_b\bigr)+o_1(r^{1-n}), \label{feq32}\\
\mathbb Q_b&\in\Omega^{n-3}(\mathbb U_{\infty})\,\,\text{has constant component functions,}
\end{align}
where $r$ denotes $\vert x\vert$. We call $\mathbb Q_b$ the tensorial magnetic charge and we denote by $\mathcal{Q}_b$ its norm induced by the Euclidean metric, that is,
\begin{equation}
   \mathcal{Q}_b=\|\mathbb Q_b\|_{\mathrm{eucl}}.
   \label{eq:qB-norm-definition}
\end{equation}
Then, the following statements hold.
\begin{enumerate}
\item[$(a)$] The constant component functions of $\mathbb Q_b$ with respect to the coordinate basis of $\Omega^{n-3}(\mathbb U_{\infty})$ are given by
\begin{equation} \label{eq:directional-magnetic-chargebis}
  (\mathbb Q_b)_{i_1\dots i_{n-3}}
   =\frac{(-1)^{n-1}}{|\SSS^{n-1}|}
     \lim\limits_{r\to +\infty}\int\limits_{\{\vert x \vert\,=\,r\}}\!\!\! b_{i_1\dots i_{n-3}j}\frac{x^{j}}{\vert x \vert}\,d\sigma_{\mathrm{eucl}}.
\end{equation}
\item[$(b)$] For any positively oriented asymptotically flat chart $(\mathsf{U}_{\infty}, \varphi'=(\mathsf{x}^1,\dots, \mathsf{x}^n))$ of order $\tau'>(n-2)/2$, we have  
\begin{align}
b&=\frac{n}{3}\,\mathsf{r}^{1-n}\,\bigl(d\mathsf{r}\wedge\mathbb Q'_b\bigr)+o_1(\mathsf{r}^{1-n}),\\
(\mathbb Q'_b)_I&=\!\!\!\!\!\sum_{\substack{ J=(j_1, \dots,j_{n-3})\\1\leqslant j_1<\dots<j_{n-3}\leqslant n}}\!\!\!\!(\mathbb Q_b)_J\, \mathrm{det}(A^I_J),\label{eq:QB-orthogonal-transformation}
\end{align}
for every $ I=(i_1, \dots,i_{n-3})$ with $1\leqslant i_1<\dots<i_{n-3}\leqslant n$. Here, $\mathsf{r}=\vert \mathsf{x}\vert$ and $A\in SO(n)$ is the matrix appearing in the change of charts $\varphi'\circ \varphi^{-1}$, which is given by
\begin{equation}
   \mathsf{x}=Ax+c+O_2(r^{1-\min\{\tau,\tau'\}})
   \label{eq:AF-transition-QB}
\end{equation}
In particular, $\mathcal{Q}_b$ is independent of the positively oriented asymptotically flat chart.
\item[$(c)$] If $\mathcal Q_b>0$, there exists an admissible smooth vector field $\mathcal Z$, related to the magnetic two--form $\beta$, such that
$$\mathcal Z^\flat=(-1)^{n-1}\star(\omega\wedge\beta),$$
for some smooth closed $(n-3)$-form $\omega$ on $M$ satisfying $\omega=\mathbb Q_b/\mathcal Q_b$ near infinity, and $\mathcal Q_{\mathcal Z}=\mathcal Q_b$.
\item[$(d)$]  Set $\mathcal{Q}=\sqrt{\mathcal{Q}_{\mathcal{E}}^2+\mathcal{Q}_{b}^2}$. If $\mathcal{Q}>0$, assume that one can choose an admissible smooth vector field $\mathcal{Z}$, satisfying $\mathcal{Q}_{\mathcal Z}=\mathcal{Q}_b$, such that the smooth vector field
$$X=-(n-1)\,\left( \frac{\mathcal{Q}_{\mathcal{E}}}{\mathcal{Q}}\, \mathcal{E}+\frac{\mathcal{Q}_{b}}{\mathcal{Q}}\,\mathcal{Z} \right)$$
meets the following conditions:
\begin{itemize}
\item[$\circ$] $X\in C^{1,\alpha}_{-1-\delta}(M;TM)$ for some $\alpha\in(0,1)$ and $\delta\in((n-2)/2,n-2)$;
\item[$\circ$] $R+2\mathrm{div}X\in L^1(M,g)$;
\item[$\circ$] $\mathrm{R}_X^{(1-n)}=\mathrm{R}+2\mathrm{div}X-(n-2)/(n-1)|X|^2\geq0$ on $M$.
\end{itemize}
Then, the inequality
\begin{equation}
  m_{\mathrm{ADM}}(g) \geq\mathcal{Q}
  \label{eq:selected-X-mass-charge-bound}
\end{equation}
holds, with equality if and only if there exists a smooth function $f\in C^{2,\alpha}_{-\delta}(M)$ such that $X=\nabla f$ and $\bigl(M,e^{-2f/(n-1)}g\bigr)$ is isometric to $(\R^n, g_{\mathrm{eucl}})$.
\end{enumerate}
\end{theorem}

Point $(a)$ of the previous theorem remains true by replacing $o_1(r^{1-n})$ with $o(r^{1-n})$ in the expression \eqref{feq32}.

We emphasize that the tensorial magnetic charge $\mathbb Q_b$ cannot be used directly in the $X$--ADM theorem, since the latter is formulated in terms of vector fields and their scalar fluxes. We therefore pair the magnetic two--form with the normalized asymptotic direction $\mathbb Q_b/\mathcal Q_b$, extend this direction to a smooth closed $(n-3)$--form $\omega$ on $M$, and use the resulting pairing to define a vector field $\mathcal Z$. The normalization is chosen so that the scalar flux of $\mathcal Z$ is precisely $\mathcal Q_b$. Now, the extension $\omega$ is not unique, and hence neither is the associated vector field $\mathcal Z$. Nevertheless, every such choice has the same asymptotic flux, namely $\mathcal Q_{\mathcal Z}=\mathcal Q_b$. Accordingly, the global curvature hypothesis in part $(d)$ is an existence assumption: it is required to hold for some suitable choice of the extension $\omega$, or equivalently of the associated vector field $\mathcal Z$.
\medskip

The paper is organized as follows. In Section \ref{sectWeighted gradient--divergence-free decomposition}, we show the existence of the weighted gradient--divergence-free decomposition of an admissible vector field satisfying the condition $(b)$ of Theorem \ref{GenPMT}, with resulting equality \eqref{feqreductionX-masstogradientmasswithoutboundary}. In Section \ref{SectionProofmain}, we prove that there exists a unique smooth positive function $u$ that satisfies \eqref{feq34}, we present its asymptotic expansion in a generic asymptotically flat chart, and finally we show the exact decomposition of the $X$--ADM mass given by \eqref{GenPMI-main}, with resulting $X$--positive mass inequality and subsequent discussion of rigidity cases. In Section \ref{Sectsharpinterval}, we construct explicit negative--mass examples that show that the range of parameters in the modified scalar curvature condition $(c)$ of Theorem \ref{GenPMT} is sharp. In Section \ref{Sectfirst consequences}, we apply Theorem \ref{GenPMT} to systems of vector fields, obtaining higher--dimensional multi--charge mass inequalities with rigidity and global alignment in the equality case, Theorem \ref{thm:Nchargepostivemass}. This theorem yields the $3$--dimensional charged positive mass theorem, Theorem \ref{crl:time-symmetric-dyonic-pmt}.
In the last section, Section \ref{SectHighmahetic}, we first show the higher--dimensional analogue of Theorem \ref{crl:time-symmetric-dyonic-pmt}, see Theorem \ref{thm:canonical-Maxwell-PMTintro}. 
Then, we associate to a magnetic $2$--form a tensor--valued charge. We show that it vanishes under suitable integrability assumptions, Proposition \ref{prop:tensorial-magnetic-chargeintro}, while it is well--defined, in a suitable sense, under a prescribed critical asymptotic expansion of its Hodge dual $b=\star \beta$, Theorem \ref{thm:selected-tensorial-mass-charge}. In the same theorem, by means of the Euclidean norm of this tensor--valued charge, we can formulate and prove a mass–charge inequality.

\section{Weighted gradient--divergence-free decomposition}\label{sectWeighted gradient--divergence-free decomposition}

We start by introducing the weighted H\"{o}lder spaces, which allow us to generalize some results about second--order linear elliptic operators on compact manifolds to the setting of asymptotically flat manifolds.

\begin{definition}[Weighted spaces]\label{def:weighted-spaces}
Let $(M,g)$ be a connected, complete, asymptotically flat manifold, of dimension $n\geq 3$.
Let $r$ be a smooth positive function equal to $|x|$ on the end. 
\begin{itemize}
\item 
For $\delta\in\R$ and $k\in\N$, the space $C^{k}_{-\delta}(M)$ consists of functions $u\in C^{k}(M)$ for which the norm 
\begin{align}
\|u\|_{ C^{k}_{-\delta}(M)}
&=\sum_{j=0}^k \sup_M \Big(r^{\delta+j}|\nabla^j u|\Big)
\end{align}
is finite.
\item 
For $\delta\in\R$, $k\in\N$, and $\alpha\in(0,1)$, the space $C^{k,\alpha}_{-\delta}(M)$ consists of functions $u\in C^{k}(M)$ for which the norm 
\begin{align}
\qquad\quad \,\,\|u\|_{ C^{k,\alpha}_{-\delta}(M)}
=\sum_{j=0}^k \sup_M \Big(r^{\delta+j}|\nabla^j u|\Big)+\sup_{\substack{B^*\!(p, \rho(p)) \\ p \in M}}
\left(\min\{r(p),r(\,\cdot\,)\}^{\delta+k+\alpha}
\frac{|\nabla^k u(p)-\nabla^k u(\,\cdot\,)|}{\mathrm{dist} (p,\,\cdot)^\alpha}\right)\!,
\end{align}
is finite. Here, $\rho(p)$ is the injectivity radius of the point $p$, the set $B^*\!(p, \rho(p))$ denotes the punctured ball centered at $p$ and radius $\rho(p)$, and for any $q\in B^*\!(p, \rho(p))$, the quantity $|\nabla^k u(p)-\nabla^k u(q)|$ is defined using parallel transport along the minimizing geodesic connecting $q$ to $p$. 
These spaces are equivalent to those introduced in \cite{Leebook}. 
They are Banach spaces independent of the asymptotically flat chart. Moreover, one can check that the multiplication is a continuous map from $C^{0,\alpha}_{-\delta_1}(M)\times C^{0,\alpha}_{-\delta_2}(M)$ to $C^{0,\alpha}_{-\delta_1-\delta_2}(M)$. Moreover, the following inclusions are true.
\begin{align}
&C^{k_1,\alpha_1}_{-\delta_1}(M)\subset C^{k_2,\alpha_2}_{-\delta_2}(M)\quad\text{if $k_1\geq k_2$, $\alpha_1\geq \alpha_2$ and $\delta_1\geq \delta_2$,}\\
&C^{k+1}_{-\delta}(M)\subset C^{k,\alpha}_{-\delta}(M)\quad\quad\,\,\text{for every $\alpha\in (0,1)$}.
\end{align}
The spaces of sections $C^{k}_{-\delta}(M; T^{l}_{m}M)$ and $C^{k,\alpha}_{-\delta}(M; T^{l}_{m}M)$ are defined analogously.
\end{itemize}
\end{definition}

We collect several useful results about second--order linear elliptic operators on weighted H\"{o}lder spaces. They can be found in \cite[Appendix A]{Leebook}, for example.

\begin{itemize}
\item[$(i)$] Let $0<\delta<n-2$ and $\alpha \in (0,1)$. The map $\Delta: C^{2,\alpha}_{-\delta}(M)\to C^{0,\alpha}_{-\delta-2}(M) $ is an isomorphism.
\item[$(ii)$] If $Z\in C^{0,\alpha}_{-1-\delta}(M;TM)$ and $h\in C^{0,\alpha}_{-2-\delta}(M)$, where $\alpha\in (0,1)$ and $0<\delta<n-2$, the second--order elliptic linear operator 
$$L=-\Delta +\nabla_{Z}+h: C^{2,\alpha}_{-\delta}(M)\to C^{0,\alpha}_{-2-\delta}(M)$$
is a Fredholm operator with zero index. 
\end{itemize}

A key result for proving our main result is the following decomposition.

\begin{proposition}[Weighted gradient--divergence-free decomposition]\label{thm:main-reduction}
Let $(M,g)$ be an $n$--dimensional, connected, complete, asymptotically flat manifold, with $n\geq3$, and let $X$ be an admissible smooth vector field. If $X\in C^{1,\alpha}_{-1-\delta}(M;TM)$ for some $\alpha\in (0,1)$ and $\delta\in \big((n-2)/2, n-2\big)$, there exist a unique smooth function $f\in C^{2,\alpha}_{-\delta}(M)$ and a unique divergence--free smooth vector field $Y\in C^{1,\alpha}_{-1-\delta}(M;TM)$ such that $X=\nabla f+Y$. Moreover, the smooth vector fields $\nabla f$ and $Y$ are orthogonal with respect to the $L^2$--inner product.
\end{proposition}

\begin{proof}
Since $\mathrm{div}(X)\in C^{0,\alpha}_{-2-\delta}(M)$ as a consequence of the assumption $X\in C^{1,\alpha}_{-1-\delta}(M;TM)$, result $(i)$ implies the existence of a unique function $f\in C^{2,\alpha}_{-\delta}(M)$ such that $\Delta f=\mathrm{div}(X)$. Notice that $f$ is smooth since $X$ is a smooth vector field on $M$, by the standard interior elliptic regularity theory. The smooth vector field $Y$ is then obtained uniquely as the difference $X-\nabla f$. Moreover, since $\nabla f$ and $Y$ belong to $C^{1,\alpha}_{-1-\delta}(M;TM)$ and since $2\delta>n-2$, it is straightforward that $\nabla f,Y\in L^2(M;TM)$. 
Now, we consider a generic asymptotically flat chart $\big(U_\infty, (x^1,\dots,x^n)\big)$, and let $M_R=M\setminus\{\vert x\vert>R\}$ be the compact region bounded by the coordinate sphere $\{\vert x\vert=R\}$. Since $\mathrm{div}Y=0$, integration by parts gives
\begin{equation}\label{feq1}
\int\limits_{M_R}g(\nabla f,Y)\,d\mu=\!\!\int\limits_{\{\vert x\vert=R\}}\!\!f\,g(Y,\nu)\,d\sigma .
\end{equation}
Therefore, by using the following asymptotic behaviors
$$
f=O(\vert x\vert^{-\delta}),
\qquad
|Y|=O(\vert x\vert^{-1-\delta}),
\qquad
\mathrm{Area}(\{\vert x\vert=R\})=O(R^{n-1}),
$$
and that $2\delta>n-2$, the right--hand side of equality \eqref{feq1} tends to zero as $R\to+\infty$. By Lebesgue's dominated convergence theorem, we then conclude that
$$\int_M g(\nabla f,Y)\,d\mu=0.$$
\end{proof}

An important consequence of the weighted gradient--divergence-free decomposition is contained in the following remark.

\begin{remark}\label{rem:Xadmequivmf}
Let $(M,g)$ be an $n$--dimensional, connected, complete, asymptotically flat manifold, with $n\geq3$, and let $X$ be an admissible smooth vector field such that $X\in C^{1,\alpha}_{-1-\delta}(M;TM)$ for some $\alpha\in (0,1)$ and $\delta\in \big((n-2)/2, n-2\big)$. Assume that $\mathrm{R}+2\,\mathrm{div}(X)\in L^1(M,g)$. By the weighted gradient--divergence-free decomposition of $X$, this assumption is equivalent to $\mathrm{R}+2\,\Delta f\in L^1(M,g)$. 
Moreover, we have
\begin{align*}
m_{X}&=\frac{1}{2(n-1)|\SSS^{n-1}|}\lim_{r\to +\infty}\int\limits_{\{\vert x \vert\,=\,r\}}\!\!\!\delta^{i}_{j}(\delta^{kl}\partial_{k}g_{il}-\delta^{kl}\partial_{i}g_{kl}+2\delta_{ik}X^k)\frac{x^{j}}{\vert x \vert}\,d\sigma_{\mathrm{eucl}}\\
&=\frac{1}{2(n-1)|\SSS^{n-1}|}\lim_{r\to +\infty}\int\limits_{\{\vert x \vert\,=\,r\}}\!\!\!\delta^{i}_{j}(\delta^{kl}\partial_{k}g_{il}-\delta^{kl}\partial_{i}g_{kl}+2\delta_{ik} \nabla f^k+2\delta_{ik}Y^k)\frac{x^{j}}{\vert x \vert}\,d\sigma_{\mathrm{eucl}}\\
&=\frac{1}{2(n-1)|\SSS^{n-1}|}\lim_{r\to +\infty}\bigg[\,\int\limits_{\{\vert x \vert\,=\,r\}}\!\!\!\delta^{i}_{j}(\delta^{kl}\partial_{k}g_{il}-\delta^{kl}\partial_{i}g_{kl}+2\delta_{i}^{k}\partial_k f)\frac{x^{j}}{\vert x \vert}\,d\sigma_{\mathrm{eucl}}+\!\!\!\int\limits_{\{\vert x \vert\,=\,r\}}\!\!\! 2g_{ij}Y^i\nu^j\,d\sigma +o(1) \,\bigg]\\
&=\frac{1}{2(n-1)|\SSS^{n-1}|}\lim_{r\to +\infty}\int\limits_{\{\vert x \vert\,=\,r\}}\!\!\!\delta^{i}_{j}(\delta^{kl}\partial_{k}g_{il}-\delta^{kl}\partial_{i}g_{kl}+2\delta_{i}^{k}\partial_k f)\frac{x^{j}}{\vert x \vert}\,d\sigma_{\mathrm{eucl}}, 
\end{align*}
where the penultimate equality follows from the observation that $\tau+\delta>n-2$ and the last one is a consequence of the divergence theorem and of $\mathrm{div}(Y)=0$ on $M$.
\end{remark}

{\em From now on, by replacing $\delta$ or $\tau$ with a smaller value if necessary, we may assume $\delta=\tau$.}

\section{Proof of main Theorem \ref{GenPMT}}\label{SectionProofmain}
Let $(M,g)$ be an $n$--dimensional, connected, complete, asymptotically flat manifold, with $n\geq3$, and let $X$ be an admissible smooth vector field satisfying $X\in C^{1,\alpha}_{-1-\delta}(M;TM)$ for some $\alpha\in (0,1)$ and $\delta\in \big((n-2)/2, n-2\big)$.
Starting from the weighted gradient--divergence-free decomposition of $X$, let us seek a smooth function $w\in  C^{2,\alpha}_{-\delta}(M)$ such that $\widetilde{X}=\nabla (f+ w)$ is an admissible smooth vector field and satisfies $$\mathrm{R}_{\widetilde{X}}^{(1-n)}=\mathrm{R}_{X}^{(1-n)}.$$
More precisely, we seek a smooth function $w\in  C^{2,\alpha}_{-\delta}(M)$ that solves the semilinear second--order partial differential equation
\begin{equation}\label{eq:Ew}
        \Delta w-a\,g(\nabla f,\nabla w)-\frac a2|\nabla w|^2
        =-a\,g(\nabla f,Y)-\frac a2|Y|^2
\end{equation}
on $M$, where $ a=(n-2)/(n-1)$.
Indeed, the identities
\begin{align}
\mathrm{R}_{\widetilde{X}}^{(1-n)}&=\mathrm{R}+2\Delta f+2\Delta w-a\left(\vert \nabla f\vert^2+\vert \nabla w\vert^2+2g(\nabla f,\nabla w)\right)\\
\mathrm{R}_{X}^{(1-n)}&=\mathrm{R}+2\Delta f-a\left(\vert \nabla f\vert^2+\vert Y\vert^2+2g(\nabla f,Y)\right)
\end{align}
are true by definition and using $\mathrm{div}(Y)=0$, and their difference vanishes exactly when \eqref{eq:Ew} holds.

In order to find the function $w$, we will show that there exists a unique smooth positive function
$$u=1+v, \qquad v\in C^{2,\alpha}_{-\delta}(M)\cap C^{\infty}(M),$$
that satisfies the second--order elliptic linear partial differential equation
\begin{equation}\label{eq:L-u-main}
        Lu:=\Delta u-a\,g\left(\nabla f,\nabla u\right)-Vu=0 
\end{equation}
on $M$ and tends to $1$ at infinity, where
\begin{equation}\label{eq:potential-V-main}
        V=\frac{a^2}{2}\,g(\nabla f,Y)+\frac{a^2}{4\,}\vert Y\vert^2.
\end{equation}
Indeed, if $u$ is as above, the smooth function
\begin{equation}\label{eq:w-def-main}
        w=-\frac{2}{a}\log u,
\end{equation}
is well--defined, immediately belongs to $C^{2,\alpha}_{-\delta}(M)$, and satisfies the equation \eqref{eq:Ew}. The last claim follows from straightforward calculations, as
\begin{align*}
\nabla w&=-\frac{2}{a}\, \frac{\nabla u}{u},\\
\Delta w&=-\frac{2}{a}\,\frac{\Delta u}{u}+\frac{2}{a} \,\frac{\vert \nabla u\vert^2}{u^2}.
\end{align*} 

In the next subsection, we will show that such a function $u$ exists and is unique.

\subsection{Existence, uniqueness and positivity of the solution $u$}

In this subsection, we study second--order linear elliptic operators $L$ of the form
\begin{equation}\label{fLoperator}
L\phi:=\Delta \phi-a\,g\left(\nabla f,\nabla \phi\right)-V\phi 
\end{equation}
for any $\phi\in C^2(M)$, where 
$$a=(n-2)/(n-1),\quad\,\, f\in C^\infty(M)\cap C^{2,\alpha}_{-\delta}(M),\quad\,\,\delta\in \big((n-2)/2, n-2\big), \quad\,\,
        V=\frac{a^2}{2}\,g(\nabla f,Y)+\frac{a^2}{4\,}\vert Y\vert^2,$$
and $Y$ is a divergence--free smooth vector field satisfying $Y\in C^{1,\alpha}_{-1-\delta}(M;TM)$. 

We notice that, since $\nabla f$ and $Y$ belong to $C^{1,\alpha}_{-1-\delta}(M;TM)$, the continuity of multiplication in weighted H\"{o}lder spaces gives $g(\nabla f,Y),\; |Y|^2\in C^{0,\alpha}_{-2-2\delta}(M)$. Consequently, $V\in C^{0,\alpha}_{-2-2\delta}(M)$. By the inclusion property of weighted H\"{o}lder spaces, $V\in C^{0,\alpha}_{-2-\delta}(M)$ as $\delta>0$. As a consequence, we can apply the weighted Fredholm theory to the operator $L$.\\
Finally, we also observe that, since $g(\nabla f,Y)$ and $|Y|^2$ are smooth functions, $V\in C^{\infty}(M)$.

A key result is the following lemma, together with its weak extension stated in Remark \ref{rmk:weak-perfect-square}.

\begin{lemma}[The perfect--square identity]\label{lem:square}
For every $\phi\in C_c^\infty(M)$,
\begin{equation}\label{eq:square-diagonal}
        -\int\limits_M e^{-af}\phi L\phi\,d\mu
        =\int\limits_M e^{-af}\left|\nabla\phi+\frac a2\phi Y\right|^2\,d\mu.
\end{equation}
More generally, for all $\psi,\eta\in C_c^\infty(M)$,
\begin{equation}\label{eq:square-bilinear}
        -\int\limits_M e^{-af}\psi L\eta\,d\mu
        =\int\limits_M e^{-af}g\!\left(
        \nabla\psi+\frac a2\psi Y,
        \nabla\eta+\frac a2\eta Y
        \right) d\mu.
\end{equation}
\end{lemma}

\begin{proof}
We start by noticing that
$$e^{-af}L\eta= \mathrm{div}(e^{-af}\nabla \eta) -e^{-af}V\eta.$$
Then, multiplying by $-\psi$ and integrating by parts, we obtain
$$
        -\int\limits_M e^{-af}\psi L\eta\,d\mu
        =\int\limits_M e^{-af}g(\nabla\psi,\nabla\eta)\,d\mu
        +\int\limits_M e^{-af}V\psi\eta\,d\mu.
$$
On the other hand, we have
\begin{align}
&\int\limits_M e^{-af}g\!\left(
        \nabla\psi+\frac a2\psi Y,
        \nabla\eta+\frac a2\eta Y
        \right)\,d\mu \\
&\quad =\int\limits_M e^{-af}g\!\left(\nabla\psi,\,\nabla\eta\right)\,d\mu
+\frac a2\int\limits_M e^{-af}g\!\left(\nabla(\psi\eta),\,Y\right)\,d\mu
+\frac{a^2}{4}\int\limits_M e^{-af}\psi\eta |Y|^2 \,d\mu\\
&\quad =\int\limits_M e^{-af}g\!\left(\nabla\psi,\,\nabla\eta\right)\,d\mu
+\frac{a^2}{2} \int\limits_M e^{-af}\psi\eta\, g(\nabla f,Y)\,d\mu
+\frac{a^2}{4}\int\limits_M e^{-af}\psi\eta |Y|^2 \,d\mu
\end{align}
where the last equality follows from integrating by parts the second integral and using $\mathrm{div}(Y)=0$. 
The last two terms give exactly $\int\limits_{M} e^{-af}V\psi\eta\,d\mu$.
\end{proof}

\begin{remark}\label{rmk:weak-perfect-square}
Identity \eqref{eq:square-bilinear} remains valid when $\psi\in W^{1,2}(M)\cap C_c(M)$ and $\eta\in C^\infty(M)$.

Indeed, let $\Omega\Subset M$ be a relatively compact open set such
that $\mathrm{supp}\psi\Subset\Omega$,
and choose a function $\chi\in C^\infty_c(M)$ satisfying $\chi\equiv1$
on an open neighborhood of $\overline{\Omega}$. By the completeness of $(M,g)$ and by the density of
$C^\infty_c(M)$ in $W^{1,2}(M)$, there exists a sequence $\{\psi_j\}_{j\in\N}\subset C^\infty_c(M)$ such that $\psi_j\to\psi$ in $W^{1,2}(M)$.
Applying then identity \eqref{eq:square-bilinear} to $\psi_j$ and $\chi\eta$, we obtain
$$
-\int\limits_M e^{-af}\psi_j L(\chi\eta)\,d\mu
=
\int\limits_M e^{-af}
g\left(
\nabla\psi_j+\frac{a}{2}\psi_jY,\,
\nabla(\chi\eta)+\frac{a}{2}\chi\eta Y
\right)
\,d\mu,
$$
Therefore, passing to the limit as $j\to+\infty$, we get
$$
-\int\limits_M e^{-af}\psi  L(\chi\eta)\,d\mu
=
\int\limits_M e^{-af}
g\left(
\nabla\psi+\frac{a}{2}\psi Y,\,
\nabla(\chi\eta)+\frac{a}{2}\chi\eta Y
\right)
\,d\mu.
$$
The desired identity then follows as $\mathrm{supp}\psi\Subset\Omega$ and $\chi\equiv1$ on an open neighborhood of $\overline{\Omega}$.
\end{remark}

This central result allows us to establish the following triviality result for the operator $L$.

\begin{lemma}[Triviality of the decaying kernel]\label{lem:kernel}
If $\phi\in C^{2,\alpha}_{-\delta}(M)$ satisfies $L\phi=0$ on $M$, then $\phi\in C^{\infty}(M)$ and $\phi =0$ everywhere.
\end{lemma}

\begin{proof}
By the smoothness of the functions $f$ and $V$ and by the standard interior elliptic regularity theory, it follows that the function $\phi$ is smooth.

We consider a generic asymptotically flat chart $\big(U_\infty, (x^1,\dots,x^n)\big)$, and
let $\chi_R$ be a cut--off function equal to $1$ on $M_R=M\setminus \{\vert x\vert> R\}$, supported in $M_{2R}=M\setminus\{\vert x\vert >2R\}$, and satisfying 
\begin{equation}
0\leqslant \chi_R\leqslant 1 \quad\text{and}\quad|\nabla^{\mathrm{eucl}}\chi_R|\leqslant C/R\quad \text{and}\quad \vert \nabla^{\mathrm{eucl}} d\chi_R\vert\leqslant C/R^2, 
\end{equation}
where $C$ is a positive constant independent of $R$, for any $R>0$ sufficiently large. We apply the identity \eqref{eq:square-diagonal} to the function $\chi_R\phi\in C^{\infty}_c(M)$, thereby obtaining
\begin{align}
&\int\limits_M e^{-af}\chi_R^2\left|\nabla\phi+\frac a2\phi Y\right|^2\,d\mu\,+\, \int\limits_M e^{-af}\Big[\phi^2 \vert \nabla \chi_R\vert^2+2\phi\chi_Rg(\nabla \chi_R,\nabla\phi)+ a\phi^2\chi_R g(\nabla \chi_R,Y)\Big]\,d\mu\,\\
&\quad= -\int\limits_M e^{-af}\phi\chi_R\Big[\chi_R L\phi+\phi  \Delta\chi_R+2 g(\nabla \chi_R,\nabla \phi)-a\phi g(\nabla \chi_R,\nabla f)\Big] \,d\mu.
\end{align}
Then, by using $L\phi=0$, we have
\begin{align*}\label{eq:cutoff-kernel-estimate}
       & \int\limits_M e^{-af}\chi_R^2\left|\nabla\phi+\frac a2\phi Y\right|^2\, d\mu\\
        &\,\,\,\leqslant\!\! \!\!\!\int\limits_{M_{2R}\setminus M_R}\!\!\!\!\!\!\!e^{-af}\Big[\phi^2|\nabla\chi_R|^2\!+\phi^2\chi_R|\Delta\chi_R|+4\vert\phi \vert\chi_R \vert \nabla \chi_R\vert\,\vert \nabla \phi\vert+a\phi^2 \chi_R \vert \nabla \chi_R\vert\, \vert Y\vert+ a\phi^2\chi_R \vert \nabla \chi_R\vert\, \vert \nabla f\vert \Big]\,d\mu\\
        &\,\,\,\leqslant \widetilde{C}\, R^{n-2-2\delta}
\end{align*}
for a positive constant $\widetilde{C}$ independent of $R$, for any $R>0$ sufficiently large. Here, we used the asymptotic behaviors in the asymptotically flat chart of the metric and related quantities.
Now, we notice that the right-hand side tends to zero as $R\to +\infty$ because $\delta>(n-2)/2$. Therefore, passing to the limit as $R\to +\infty$, Fatou's lemma yields
\begin{equation}
 \int\limits_M e^{-af}\left|\nabla\phi+\frac a2\phi Y\right|^2\, d\mu=0
\end{equation}
which in turn implies
\begin{equation}\label{eq:first-order-kernel}
        \nabla\phi+\frac a2\phi Y=0
\end{equation}
everywhere (taking into account that the function $\phi$ and the vector field $Y$ are smooth). Let us show that this identity, together with the
decay of $\phi$, implies that $\phi\equiv 0$.
Let $\gamma\colon[0,T]\to M$ be an arbitrary smooth curve.
By identity \eqref{eq:first-order-kernel}, the function $h(t)=\phi(\gamma(t))$, defined in the interval $[0,T]$, satisfies the ordinary differential
equation
$$h'(t)=-\frac{a}{2}\,h(t)\,g\bigl(Y, \gamma'(t)\bigr).$$
Hence, we get
\begin{equation}\label{eq:phi-along-curves}
\phi(\gamma(t))=h(t)
=h(0)\,e^{-\frac{a}{2}
\int_0^t g(Y, \gamma'(\tau))\,d\tau}
=\phi(\gamma(0))\,e^{-\frac{a}{2}
\int_0^t g(Y, \gamma'(\tau))\,d\tau}.
\end{equation}
As a consequence, if $\phi$ vanishes at one point of $M$, then the equality
\eqref{eq:phi-along-curves} shows that it vanishes along every smooth curve starting at that point. The connectedness of $M$, which implies that any two points of $M$ can be joined by such a curve, yields $\phi\equiv 0$. It remains to rule out the possibility that $\phi$ is nowhere vanishing. Suppose, by contradiction, that this is the case. 
For a fixed asymptotically flat chart $\big(U_\infty, (x^1,\dots,x^n)\big)$, we choose a smooth curve $\gamma\colon[0,+\infty)\longrightarrow M_{\infty}$ which coincides with a radial coordinate curve. With this choice, from $Y\in C^{1,\alpha}_{-1-\delta}(M;TM)$ it follows that $$\vert g(Y,\gamma'(t))\vert \leqslant \widehat{C}(1+t)^{-1-\delta} $$ for every $t\in [0,+\infty)$, where $\widehat{C}$ is a positive constant independent of $t$. Since $\delta>0$, we have
$$\int\limits_0^{+\infty}\bigl|g(Y,\gamma'(t))\bigr|\,dt<+\infty.$$
Then, arguing as before and passing to the limit in \eqref{eq:phi-along-curves}, we conclude that
$$
\lim_{t\to+\infty}
|\phi(\gamma(t))|
=
|\phi(\gamma(0))|
\,e^{-\frac{a}{2}
\int_0^{+\infty} \!\!g(Y, \gamma'(\tau))\,d\tau}
>0.$$
On the other hand, since $\phi\in C^{2,\alpha}_{-\delta}(M)$, we have $\phi \to 0$ at infinity, which is a contradiction. Therefore, $\phi=0$ on all of $M$.
\end{proof}

Now, we are able to establish the following result of existence and uniqueness.

\begin{theorem}[Invertibility and existence]\label{thm:invertibility}
The operator
\begin{equation}\label{eq:L-map}
        L:C^{2,\alpha}_{-\delta}(M)\longrightarrow C^{0,\alpha}_{-\delta-2}(M)
\end{equation}
is an isomorphism. Consequently, there exists a unique function $v\in C^{2,\alpha}_{-\delta}(M)\cap C^{\infty}(M)$ solving
\begin{equation}\label{eq:v-equation}
        Lv=V,
\end{equation}
and the smooth function $u=1+v$ satisfies
\begin{equation}\label{eq:u-existence}
        Lu=0  \qquad \text{and} \qquad u\to1 \quad \text{at infinity}.
\end{equation}
\end{theorem}

\begin{proof}
The operator $L$ is a lower--order perturbation of the Laplacian operator, having $\nabla f\in C^{0,\alpha}_{-1-\delta}(M;TM)\cap  C^{\infty}(M;TM)$ and $V\in C^{0,\alpha}_{-2-\delta}(M)\cap C^\infty(M)$ with $\delta\in  \big((n-2)/2, n-2\big)$. 
Weighted elliptic Fredholm theory implies that $L:C^{2,\alpha}_{-\delta}(M)\longrightarrow C^{0,\alpha}_{-\delta-2}(M)$ is a Fredholm operator with index zero (see $(ii)$ at the beginning of Section \ref{sectWeighted gradient--divergence-free decomposition}). Lemma \ref{lem:kernel} shows that its kernel is trivial, hence, $L$ is an isomorphism. Consequently, there is a unique function $v\in C^{2,\alpha}_{-\delta}(M)$ solving $Lv=V$, as $V\in C^{0,\alpha}_{-\delta-2}(M)$. Notice that the function $v$ is smooth, by the standard interior elliptic regularity theory. Finally, we observe that $u=1+v$ satisfies $Lu=L1+Lv=-V+V=0$.
\end{proof}

The last part of this subsection aims to show that the above function $u$ is positive on $M$.

\begin{proposition}[Strict positivity]\label{prop:positivity}
The solution $u$ of \eqref{eq:u-existence} is strictly positive on $M$.
\end{proposition}

\begin{proof}
We recall that the solution $u$ is smooth. Let us consider the negative part $u_-$ of $u$, which is given by $u_-=\max\{-u,0\}$. Notice that: $u_-$ is locally Lipschitz on $M$, therefore, $u_-$ is differentiable almost everywhere on $M$, by Rademacher's theorem; it has $\nabla u_-=-\mathbb{I}_{\{u<0\}}\nabla u$ almost everywhere on $M$ (where $\mathbb{I}_E$ denotes the characteristic function of a set $E$); and it has compact support since $u\to 1$ at infinity. By the first and third property of $u_-$, we can apply Remark~\ref{rmk:weak-perfect-square} with $\psi=u_-$ and $\eta=u$. Using $Lu=0$ on $M$, we then obtain
\begin{align*}
0=-\int\limits_M e^{-af}u_-\,Lu\,d\mu=
\int\limits_M e^{-af}
g\left(
\nabla u_-+\frac{a}{2} u_- \,Y,\,
\nabla u+\frac{a}{2} u Y
\right)
\,d\mu=-\!\!\!\int\limits_{\{u<0\}}\!\!\!
e^{-af}
\left|
\nabla u+\frac{a}{2}uY
\right|^2
\,d\mu,
\end{align*}
which implies that, if $\{u<0\}$ is not empty, then
$$\nabla u+\frac{a}{2}uY=0\quad\quad \text{on $\{u<0\}$}.$$
Suppose, by contradiction, that $\{u<0\}$ is not empty. The continuity of $u$ and the convergence of $u$ to $1$ at infinity guarantee that this set is both open and relatively compact. Thus, there exists a connected component $\Omega$ such that $\partial \Omega\subset \{u=0\}$. This leads to the existence of a smooth curve $\gamma:[0,T]\to M$ joining a point $p\in \Omega$ with a point $q\in \partial\Omega$ and such that $\gamma |_{[0,T)}$ is contained in $\Omega$. Arguing as in the proof of Lemma \ref{lem:kernel}, we get 
\begin{equation}\label{eq:u-along-curves}
u(\gamma(t))=u(\gamma(0))\,e^{-\frac{a}{2}
\int_0^t g(Y, \gamma'(\tau))\,d\tau}=u(p)\,e^{-\frac{a}{2}
\int_0^t g(Y, \gamma'(\tau))\,d\tau}
\end{equation}
for every $t\in [0,T)$. Passing to the limit as $t\to T$, we conclude 
$$0=u(q)=u(p)\,e^{-\frac{a}{2}
\int_0^T g(Y, \gamma'(\tau))\,d\tau}<0,$$
which is a contradiction, thus, $u\geq 0$ on $M$. It remains to prove that $u$ is positive on $M$.
Since $M$ is connected and $u\geq 0$, the Harnack inequality (applied to appropriate coordinate domains) implies that either $u> 0$ everywhere or $u\equiv 0$. The latter is impossible because $u\to 1$ at infinity.
\end{proof}

\subsection{Improvement of asymptotic expansion of the function $u$}

This subsection is dedicated to improving the asymptotic expansion of $v$ and, consequently, of $u$ at infinity.

We start by recalling that since $\nabla f$ and $Y$ belong to $C^{1,\alpha}_{-1-\delta}(M;TM)$, it holds that $V\in C^{0,\alpha}_{-2-2\delta}(M)$. Moreover, because $v$ belongs to $C^{2,\alpha}_{-\delta}(M)$, the equation $Lv=V$ implies that $\Delta v \in C^{0,\alpha}_{-2-2\delta}(M)$ by the properties of H\"{o}lder spaces.

Choose $\beta\in \big(n-2, 2\delta\big)$ such that $-\beta \notin\Lambda$, where $\Lambda$ is the exceptional set $\mathbb{Z}\setminus(2-n,0)$. Since $\Delta v \in C^{0,\alpha}_{-2-2\delta}(M)\subset C^{0,\alpha}_{-2-\beta}(M)$ and $v\in C^{2,\alpha}_{-\delta}(M)$, we may apply \cite[Corollary A.38]{Leebook}. Notice that the largest exceptional weight in $(-\beta,-\delta)$ is $2-n$. Then, by \cite[Corollary A.38]{Leebook} (see also Remark A.39), there exists $B\in\R$ and $\gamma>0$ such that
\begin{equation}\label{feq4}
v=B\,\vert x\vert^{2-n}\!+O_2(\vert x\vert^{2-n-\gamma})
\end{equation}
in an arbitrary asymptotically flat chart $\big(M_{\infty}, (x^1,\dots,x^n)\big)$. 

This asymptotic expansion of $v$ implies that 
\begin{equation}\label{betterexpu}
u=1+B\,\vert x\vert^{2-n}\!+O_2(\vert x\vert^{2-n-\gamma}).
\end{equation}
We claim that 
\begin{equation}\label{eq:B-formula}
        B=-\frac{1}{(n-2)|\mathbb{S}^{n-1}|}
        \int\limits_M e^{-af}\left|\nabla u+\frac a2 uY\right|^2\,d\mu\,\leqslant \,0.
\end{equation}
In order to show this equality, we denote by $M_R$ the compact region $M\setminus \{\vert x\vert> R\}$ bounded by a large coordinate sphere $\{\vert x\vert=R\}$. By using the equation of $u$, the definition of $V$ and that $\mathrm{div}(Y)=0$, we have 
$$0=-e^{-af}uLu=-\mathrm{div}\left[ e^{-af}u\left(\nabla u+\frac{a}{2}u Y\right)\right]+e^{-af}\left|\nabla u+\frac a2uY\right|^2$$
and hence
\begin{align}
        \int\limits_{M_R}e^{-af}\left|\nabla u+\frac a2uY\right|^2d\mu
        &=\!\!\!\int\limits_{\{\vert x\vert =R\}}\!\!\!e^{-af}u\, g\!\left(\nabla u+\frac{a}{2}u Y,\nu\right)\,d\sigma\\
        &=\!\!\!\int\limits_{\{\vert x\vert =R\}}\!\!\!e^{-af}u\, g\!\left(\nabla u,\nu\right)\,d\sigma+\frac{a}{2}\!\!\!\int\limits_{\{\vert x\vert =R\}}\!\!\!(e^{-af}u^2-1)\, g\!\left(Y,\nu\right)\,d\sigma,\label{feq5}
\end{align}
where the last equality follows from $\mathrm{div}(Y)=0$.
Recalling that
\begin{align}
g_{ij}&=\delta_{ij}+O(\vert x\vert^{-\delta}),\label{feq26}\\
g^{ij}&=\delta^{ij}+O(\vert x\vert^{-\delta}),\label{feq27}\\
\nu^i&=\nu^i_{\text{eucl}}+O(\vert x\vert^{-\delta}),\label{feq28}\\
d\sigma&=\left[1+O(\vert x\vert^{-\delta})\right]\,d\sigma_{\text{eucl}},\label{feq29}
\end{align}
see for instance \cite[Section 1.4]{phdthesisOro}, and obtaining from $f=O(\vert x\vert^{-\delta})$ that $e^{-af}=1+O(\vert x\vert^{-\delta})$ we deduce that
\begin{align}
(e^{-af}u^2-1)&=O(\vert x\vert^{-\delta}),\\
\int\limits_{\{\vert x\vert =R\}}\!\!\!(e^{-af}u^2-1)\, g\!\left(Y,\nu\right)\,d\sigma&=\int\limits_{\{\vert x\vert =R\}}\!\!\!O(\vert x\vert^{-1-2\delta})\,d\sigma_{\text{eucl}}=O(R^{n-2-2\delta}),
\end{align}
and that
\begin{align}
g\!\left(\nabla u,\nu\right)&=\delta^i_{j}\partial_i u\,\nu^j_{\text{eucl}}+O(\vert x\vert^{-1-2\delta}),\\
\int\limits_{\{\vert x\vert =R\}}\!\!\!e^{-af}u\, g\!\left(\nabla u,\nu\right)\,d\sigma&=\int\limits_{\{\vert x\vert =R\}}\!\!\!\delta^i_{j}\partial_i u\,\nu^j_{\text{eucl}}\,d\sigma_{\text{eucl}}+O(R^{n-2-2\delta}).
\end{align}
Using the better asymptotic expansion of $u$, given by formula \eqref{betterexpu}, we then get
$$\int\limits_{\{\vert x\vert =R\}}\!\!\!\delta^i_{j}\partial_i u\,\nu^j_{\text{eucl}}\,d\sigma_{\text{eucl}}=-B(n-2)|\mathbb{S}^{n-1}|+O(R^{-\gamma}).$$
Thus, passing to the limit as $R\to +\infty$ and taking into account that $2\delta>n-2$, we conclude that
\begin{equation}
        B=-\frac{1}{(n-2)|\mathbb{S}^{n-1}|}
        \int\limits_M e^{-af}\left|\nabla u+\frac a2 uY\right|^2\,d\mu.
\end{equation}

We conclude this subsection with the following remark of rigidity.

\begin{remark}\label{RigB=0}
The following conditions are equivalent.
\begin{enumerate}
\item $B= 0$;
\item $Y\equiv0$;
\item $u\equiv1$.
\end{enumerate}
Indeed,  if $B=0$, formula \eqref{eq:B-formula} gives
\begin{equation}\label{eq:Bzero-firstorder}
        \nabla u+\frac a2uY=0
\end{equation}
everywhere.
Since $u>0$, it follows that $Y=-(2/a)\nabla\log u$. Taking the divergence and using the fact that $\mathrm{div}Y=0$, we obtain that $\Delta\log u=0$ on $M$. Moreover, since $u\to1$ implies that $\log u\to0$ at infinity, the maximum principle ensures that $\log u$ is constant, and this constant must be $0$, in particular, $u\equiv 1$ and $Y\equiv0$.\\
If $u\equiv 1$, asymptotic expansion \eqref{betterexpu} of $u$ implies that $B=0$, which in turn yields $Y\equiv 0$ by identity \eqref{eq:Bzero-firstorder}.\\
If $Y\equiv0$, then $V\equiv0$. Since $u$ solves $Lu=0$ on $M$ and $u\to1$ at infinity, the maximum principle guarantees that $u\equiv1$, which in turn gives $B=0$.
\end{remark}

\subsection{$X$--positive mass inequality}
By Theorem \ref{thm:invertibility} and Proposition \ref{prop:positivity}, it follows that the function $w=-(2/a)\log u$ is well defined and smooth. Indeed, $u=1+v$, with $v\in C^{2,\alpha}_{-\delta}(M)\cap C^{\infty}(M)$, and $u$ is strictly positive on $M$. Furthermore, by standard composition properties of weighted H\"{o}lder spaces, $w\in C^{2,\alpha}_{-\delta}(M)$. This fact along with $f\in C^{2,\alpha}_{-\delta}(M)\cap C^{\infty}(M)$ guarantee that $\widetilde X=\nabla(f+w)$ is an admissible smooth vector field. We also find that
\begin{equation}\label{feq3}
\mathrm{R}_{\widetilde{X}}^{(1-n)}=\mathrm{R}_{X}^{(1-n)},
\end{equation}
since $w$ solves 
\begin{equation}
        \Delta w-a\,g(\nabla f,\nabla w)-\frac a2|\nabla w|^2
        =-a\,g(\nabla f,Y)-\frac a2|Y|^2
\end{equation}
on $M$, as a direct consequence of the fact that $Lu=0$ on $M$. Finally, the equality \eqref{feq3}
and the assumption $\mathrm{R}+2\mathrm{div}X\in L^1(M,g)$ imply $\mathrm{R}+2\mathrm{div}\widetilde{X}\in L^1(M,g)$, as $\vert\widetilde X\vert^2=O(r^{-2-2\delta})$ and $\vert X\vert^2=O(r^{-2-2\delta})$ with $\delta>(n-2)/2$.

Let us observe that, by definition, we have
\begin{equation}\label{feq9}
\mathrm{R}_X^{(1-n)}
=
\mathrm{R}_X^{(k)}
+
\left(
\frac{1}{k}+\frac{1}{n-1}
\right)|X|^2,
\end{equation}
and the last term on the right--hand side is nonnegative as $1/k+1/(n-1)\geq 0$ for any $k\in\mathbb{R}\setminus(1-n,0]$. Thus, the assumption $\mathrm{R}_X^{(k)}\geq 0$ implies $ \mathrm{R}_X^{(1-n)}\geq 0$, which in turn gives $\mathrm{R}_{\widetilde{X}}^{(1-n)}\geq 0$. Consequently, considering the metric $$\widetilde{g}=e^{-\frac{2(f+w)}{n-1}}\,g,$$
by \cite[Lemma 2.4 and Lemma 2.5]{LaLoSa}, the Riemannian $n$--manifold $(M,\widetilde{g})$ is complete, asymptotically flat, and has 
\begin{equation}\label{feq8}
\widetilde{\mathrm{R}}=e^{\frac{2(f+w)}{n-1}}\,\mathrm{R}_{\widetilde{X}}^{(1-n)}\geq 0.
\end{equation}
Moreover, Law, Lopez and Santiago showed that 
$$m_{\mathrm{ADM}}(\widetilde{g})=\frac{1}{2(n-1)|\SSS^{n-1}|}\lim_{r\to +\infty}\int\limits_{\{\vert x \vert\,=\,r\}}\! \delta^{i}_{j}\left[\delta^{kl}\partial_{k}g_{il}-\delta^{kl}\partial_{i}g_{kl}+2\delta_{i}^{k}\partial_k (f+w)\right]\frac{x^{j}}{\vert x \vert}\,d\sigma_{\mathrm{eucl}}=m_{\widetilde{X}}.$$
Therefore, the positive mass theorem in all dimensions \cite{BrendleWang2026, KhuriWangWang} yields $ m_{\widetilde{X}}\geq 0$. At this point, let us use the improved asymptotic expansion of the function $u$, given by formula \eqref{betterexpu}, to establish that
$$\partial_i w=-\frac{2}{au}\partial_iu=-\frac{2}{a}\Big[1+O(\vert x\vert^{2-n})\Big]\Big[B(2-n)\frac{x^i}{\vert x\vert^n}+O(\vert x\vert^{1-n-\gamma})\Big]$$
and, consequently, that
$$\int\limits_{\{\vert x \vert\,=\,r\}}\! \delta^{i}_{j}\partial_i w\frac{x^{j}}{\vert x \vert}\,d\sigma_{\mathrm{eucl}}=-\frac{2}{a}\int\limits_{\{\vert x \vert\,=\,r\}}\! \left[B(2-n)\vert x\vert^{1-n}+o(\vert x\vert^{1-n})\right]\,d\sigma_{\mathrm{eucl}}=2a^{-1}B(n-2)\,|\mathbb{S}^{n-1}|+o(1).$$
Together with Remark \ref{rem:Xadmequivmf}, this result implies the following relation between the masses $m_{\widetilde{X}}$ and $m_{X}$: 
\begin{equation}
m_{\widetilde{X}}=m_{X}+2B,
\end{equation}
as $a=(n-2)/(n-1)$. Consequently, by equality \eqref{eq:B-formula}, we conclude
\begin{equation}\label{feq7}
m_X=m_{\widetilde{X}}-2B=m_{\widetilde{X}}+\frac{2}{(n-2)|\mathbb{S}^{n-1}|}
        \int\limits_M e^{-af}\left|\nabla u+\frac a2 uY\right|^2\,d\mu\geq 0.
\end{equation}

\subsection{$X$--positive mass rigidity} In this subsection, $(M,g)$ is a Riemannian $n$--manifold satisfying the assumptions of Theorem \ref{GenPMT} and having $m_X=0$. By the equality \eqref{feq7}, we deduce that $m_{\widetilde{X}}=0$ and $B=0$. From $B=0$, by Remark \ref{RigB=0} it follows that $u\equiv1$ and $Y\equiv0$. A consequence of $u\equiv 1$ is $w\equiv 0$. 
By using then that $Y\equiv0$, we also obtain $X=\nabla f$. Now, recalling that $m_{\mathrm{ADM}}(\widetilde{g})=m_{\widetilde{X}}=0$ and since $w\equiv 0$, the rigidity part of the classical positive mass theorem \cite{BrendleWang2026, KhuriWangWang} implies that $(M,\widetilde{g}=e^{-\frac{2f}{n-1}}\,g)$ is isometric to $(\R^n,g_{\text{eucl}})$, which concludes the proof of case $(2)$. 
Under the stronger assumption that $\mathrm{R}_X^{(k)} \geqslant 0$ with $k\in \mathbb{R}\setminus [1-n,0]$, the flatness of $(M,\widetilde{g})$ and equality \eqref{feq8} imply $\mathrm{R}_{\widetilde{X}}^{(1-n)}\equiv 0$, which by equality \eqref{feq3} yields $\mathrm{R}_{X}^{(1-n)}\equiv 0$. Then, equality \eqref{feq9} leads to $X\equiv 0$. Since $X=\nabla f$ and $X\equiv 0$, we have $\nabla f\equiv 0$. Thus, $f$ is constant and this constant must be $0$ as $f\to 0$ at infinity. Consequently, $\widetilde{g}=g$ on $M$ and $(M,g)$ is isometric to $(\R^n,g_{\text{eucl}})$. \\
If $X\equiv 0$ and $(M,g)$ is isometric to $(\R^n,g_{\text{eucl}})$, a direct computation yields $m_X(g)=m_{\mathrm{ADM}}(g)=0$. Similarly, if $X=\nabla f$ and $(M,e^{-2f/(n-1)}g)$ is isometric to $(\R^n,g_{\mathrm{eucl}})$, a direct computation shows that $m_X(g)=m_{\mathrm{ADM}}(e^{-2f/(n-1)}g)=0$.

\section{Sharpness of the admissible range of $k$ in Theorem \ref{GenPMT}}\label{Sectsharpinterval}

The exclusion of the interval $k\in(1-n,0)$ in the $X$--positive mass theorem is necessary. Indeed, in this section, we will show that, for any $n\geqslant 3$ and $1-n<k<0$, there exist a smooth, connected, complete, asymptotically flat Riemannian manifold $(\R^n,g)$ and a smooth compactly supported vector field $X$ such that
\begin{align}
\mathrm{R}+&2\mathrm{div} X\in L^1(\R^n,g),\\
\mathrm{R}_X^{(k)}=\mathrm{R}+2\mathrm{div}& X-\left(1+1/k\right)|X|^2\geqslant 0 \,\,\text{on $\R^n$},\\
&m_X(g)<0.
\end{align}
In particular, the positivity conclusion of the $X$--positive mass theorem
generally fails for every $k\in(1-n,0)$.

 We fix a nondecreasing function $m\in C^\infty([0,+\infty))$ such that
$$
0\leqslant m\leqslant 1,\qquad
m(r)=0 \quad \text{for } r\leqslant \frac14,
\qquad
m(r)=1 \quad \text{for } r\geqslant \frac12,
$$
and define
$$
\zeta(x)=\int\limits_{|x|}^{+\infty} \!m(s)s^{1-n}\,ds
$$
for every $x\in \R^{n}\setminus\{0\}$. We observe that, since $m\equiv0$ on the interval $[0, 1/4]$, the function $\zeta$ is constant in a neighborhood of the origin, hence, it extends to a smooth radial function on all of $\mathbb{R}^n$. Moreover, in polar coordinates, we have
$$
-\Delta_{\mathrm{eucl}}\zeta
=r^{1-n} m'(r),
$$
therefore, the function $\rho$, defined by $\rho=-\Delta_{\mathrm{eucl}}\zeta$, satisfies the following properties:
$$
\rho\in C_c^\infty(\R^n),\qquad \mathrm{supp}\rho\subset B_1(0),\qquad
\rho\geqslant0,\qquad
$$
In addition, $\zeta>0$ on $\mathbb{R}^n$ and $\zeta(x)=\vert x\vert^{2-n}/(n-2)$ for every $x\in \R^n:\vert x\vert \geqslant 1/2$, as $m\equiv1$ for $r\geqslant1/2$.
Since $\zeta$ is a positive smooth function that vanishes at infinity, there exists $\lambda_0>0$ such that the smooth function $u=1-\lambda_0\zeta$ is positive on $\R^n$, allowing us to introduce the conformally flat metric
\begin{equation}\label{feq31}
h=u^{\frac{4}{n-2}}g_{\text{eucl}}.
\end{equation}
Since $u$ is bounded above and bounded away from zero, $h$ is uniformly
equivalent to $g_{\text{eucl}}$. Consequently, the Euclidean distance and the $h$--distance are bi--Lipschitz, and $h$ is complete. 
Moreover, the fact that $\zeta(x)=\vert x\vert^{2-n}/(n-2)$ on $\{\vert x\vert \geqslant 1/2\}$ yields that
$u=1-\lambda_0\vert x\vert^{2-n}/(n-2)$ therein, which in turn implies that $h$ is asymptotically flat of order
$n-2$. Finally, we have
\begin{align}
m_{\mathrm{ADM}}(h)&=\frac{1}{2(n-1)|\SSS^{n-1}|}\lim_{r\to +\infty}\int\limits_{\{\vert x \vert\,=\,r\}}\! \sum_{i,j=1}^n(\partial_{i}h_{ji}-\partial_{j}h_{ii})\frac{x^{j}}{\vert x \vert}\,d\sigma_{\mathrm{eucl}}\\
&=\lim_{r\to +\infty} \Big(\!-\frac{2\lambda_0}{(n-2)}\, u^{\frac{6-n}{n-2}}\Big)=-\frac{2\lambda_0}{(n-2)}<0,
\end{align}
and
\begin{align}\label{feq30}
\mathrm{R}_h=-\frac{4(n-1)}{(n-2)}\,u^{-\frac{n+2}{n-2}}\Delta_{\text{eucl}}\,u=-\frac{4(n-1)\lambda_0}{(n-2)}\,u^{-\frac{n+2}{n-2}}\,\rho\leqslant 0 \,\,\,\text{on $\R^n$}.
\end{align}
The next step is to introduce a further conformal change in order to adjust the modified scalar curvature while leaving the asymptotic behavior unchanged. 
We then choose a smooth cutoff function satisfying
$$
0\leqslant \chi\leqslant 1\qquad \chi = 1\,\,\, \text{on $\{\vert x\vert <2\}$}\qquad \mathrm{supp}\chi\subset \{\vert x\vert <3\},
$$
and consider on $\R^n$ the smooth function $f$, given by $$f(x)=\lambda\chi(x)x^1.$$ Here, the constant $\lambda$ is positive and will be chosen suitably below.
The function $f$ is compactly supported. Accordingly, the smooth vector field $$X=\nabla^g \!f$$
is admissible, and the conformal metric $$g=e^{\frac{2f}{n-1}}h$$ is complete and coincides with the asymptotically flat metric $h$ outside a compact set.
Now, from $h=e^{-\frac{2f}{n-1}}g$, it follows that
\begin{align}
\mathrm{R}_h&=e^{\frac{2f}{n-1}}\left(\mathrm{R}_g+2\Delta_gf-\frac{n-2}{n-1}\,|\nabla^g\!f|_g^2\right)\\
&=e^{\frac{2f}{n-1}}\left[\mathrm{R}_g+2\Delta_gf-\left(1+\frac{1}{k}\right)|\nabla^gf|_g^2+\left(\frac{1}{k}+\frac{1}{n-1}\right)|\nabla^g\!f|_g^2\right].
\end{align}
This, together with equality \eqref{feq30}, yields
\begin{align}
\mathrm{R}_{g,X}^{(k)}&=\mathrm{R}_g+2\Delta_gf-\left(1+\frac{1}{k}\right)|\nabla^g\!f|_g^2=e^{-\frac{2f}{n-1}}\mathrm{R}_h-\left(\frac{1}{k}+\frac{1}{n-1}\right)|\nabla^g\!f|_g^2\\
&=e^{-\frac{2f}{n-1}}\bigg[ -\frac{4(n-1)\lambda_0}{(n-2)}\,u^{-\frac{n+2}{n-2}}\,\rho -\left(\frac{1}{k}+\frac{1}{n-1}\right)|\nabla^h\!f|_h^2\, \bigg]\\
&=u^{-\frac{n+2}{n-2}}\,e^{-\frac{2f}{n-1}}\Big[ -\frac{4(n-1)\lambda_0}{(n-2)}\,\rho -\left(\frac{1}{k}+\frac{1}{n-1}\right)u|\nabla^{\text{eucl}}f|_{\text{eucl}}^2\,\Big], \label{eq:scalar-curvature-h}
\end{align}
where, in the last equality, we used identity \eqref{feq31}.
Notice that the constant $(1/k)+(1/(n-1))$ is negative because $1-n<k<0$. Thus, outside $\mathrm{supp}\rho$, formula
\eqref{eq:scalar-curvature-h} gives $\mathrm{R}_{g,X}^{(k)}\geqslant 0$, as $u>0$ on $\R^n$. On $\mathrm{supp}\rho\subset B_1(0)$, since the cutoff function $\chi$ is
identically equal to $1$, we have $f=\lambda x^1$ and
$$\mathrm{R}_{g,X}^{(k)}=u^{-\frac{n+2}{n-2}}\,e^{-\frac{2f}{n-1}}\Big[ -\frac{4(n-1)\lambda_0}{(n-2)}\,\rho -\lambda^2\!\left(\frac{1}{k}+\frac{1}{n-1}\right)u\Big].$$
Thus, choosing $\lambda>0$ such that
\begin{equation}
\label{eq:lambda-choice}
\lambda \geqslant \sqrt{\frac{\frac{4(n-1)\lambda_0}{n-2}\,\max\limits_{\R^n}\rho}{ -\!\left(\frac{1}{k}+\frac{1}{n-1}\right)\,\min\limits_{\R^n}u}},
\end{equation}
we obtain $\mathrm{R}_{g,X}^{(k)}\geqslant 0$ also on  $\mathrm{supp}\rho$. We emphasize that $\lambda$ is well defined because $\rho\in C_c^\infty(\R^n)$ is nonnegative and not identically zero, while $u$ is smooth, positive, and approaches $1$ at infinity. In conclusion, $\mathrm{R}_{g,X}^{(k)}\geqslant 0$ on all of $\R^n$ for the vector field $X$ and the metric $g$ given above.\\
We verify the remaining assumption that $\mathrm{R}_g+2\mathrm{div}_gX\in L^{1}(\R^n,g)$. 
We observe that $f=0$ on the set $\{|x|> 3\}$; consequently, therein, $X=0$ and $g=h$. The latter equality implies that $\mathrm{R}_g=\mathrm{R}_h$, which is identically zero on this set, since it lies outside the support of $\rho$. We then conclude that $\mathrm{R}_g+2\mathrm{div}_gX$ is compactly supported and thus belongs to $L^1(\R^n,g)$.\\
Finally, since $X=0$ and $g=h$ on $\{|x|> 3\}$, it follows $$m_X(g)=m_{\mathrm{ADM}}(g)=m_{\mathrm{ADM}}(h)<0,$$ 
as established above.

\section{First consequences of the $X$--positive mass theorem}\label{Sectfirst consequences}

The $X$--positive mass theorem is formulated for a single vector field, but it can be extended to systems carrying several charges. Indeed, when the total charge vector is nonzero, its direction selects a canonical linear combination of the underlying fields whose scalar charge equals the Euclidean norm of the full charge vector. The $N$--charge positive mass inequality then follows from a single  application of the critical $X$--positive mass theorem, and the equality case forces all fields to align globally with the selected linear combination.

\begin{proof}[Proof of Theorem \ref{thm:Nchargepostivemass}]
Since 
\begin{equation}\label{eq:n-dimensional-N-charge-dec-expandedintrobis}
\mathrm{R}\geqslant (n-1)(n-2)\bigl(|\mathcal{E}_1|^2+\dots+| \mathcal{E}_N|^2\bigr)+2(n-1)\sqrt{(\mathrm{div} \mathcal{E}_1)^2+\dots+(\mathrm{div} \mathcal{E}_N)^2}\geqslant0
\end{equation} on $M$,
by the positive mass theorem \cite{BrendleWang2026, KhuriWangWang}, we may assume $\mathcal{Q}> 0$.
For every $b=(b^1,\ldots,b^N)\in\SSS^{N-1}$, we define
$$
 Z_b=b^i \mathcal{E}_i \qquad \text{and} \qquad  \rho_b=\mathrm{div}Z_b=b^i\rho_{\mathcal{E}_i}.
$$
Before proceeding, we remark that $Z_b$ is a smooth admissible vector field on $M$ such that
\begin{equation}\label{feq14}
 Z_b\in C^{1,\alpha}_{-1-\delta}(M;TM)\qquad \text{and}\qquad \mathrm{div}(Z_b)\in L^1(M,g).
 \end{equation}
First, we set $\underline{b}_1=b$. Then, we extend $\{\underline{b}\}$ to an orthonormal basis $(\underline{b}_1,\underline{b}_2,\dots, \underline{b}_N)$ of $\R^N$ and define $W_\alpha=b_{\alpha}^i\mathcal{E}_i$ for every $\alpha\in \{2,\dots, N\}$.
Since $\sum_{k=1}^Nb_i^k b_j^k=\delta_{ij}$ for every $i,j\in\{1,\dots, N\}$ implies $\sum_{k=1}^N b^i_kb^j_k=\delta^{ij}$ for all $i,j\in\{1,\dots, N\}$, we have
\begin{equation}\label{feq18}
 \mathcal{E}_i=\sum_{j,k=1}^N b^i_k b^j_k\mathcal{E}_j=\sum_{j=1}^N b^i_1 b^j_1\mathcal{E}_j+\sum_{j=1}^N\sum_{\alpha=2}^N b^i_\alpha b^j_\alpha\mathcal{E}_j=b^i_1Z_b+\sum_{\alpha=2}^Nb^i_\alpha W_{\alpha},
 \end{equation}
and also,
\begin{align}
 \sum_{i=1}^N|\mathcal{E}_i|^2&=\left(\sum_{i=1}^N (b_1^i)^2\right) \vert Z_b\vert^2+2\sum_{\alpha=2}^N\left(\sum_{i=1}^N b_1^ib_{\alpha}^i \right)g(Z_b, W_\alpha)+\sum_{\alpha,\beta=2}^N\left(\sum_{i=1}^Nb^i_\alpha b^i_\beta\right) g(W_\alpha, W_{\beta})\\
 &=|Z_b|^2+\sum_{\alpha=2}^N|W_\alpha|^2.\label{eq:orthogonal-energy}
\end{align}
Now, the Cauchy--Schwarz inequality and the $N$--charge dominant energy condition \eqref{eq:n-dimensional-N-charge-decintro} yield
$$
 |\rho_b|\,\leqslant\,\sqrt{\rho_{\mathcal{E}_1}^2+\dots+\rho_{\mathcal{E}_N}^2}\,\leqslant\,\mu.
$$
Combining identity \eqref{eq:generalized-matter-densityintro} with this last result and equality \eqref{eq:orthogonal-energy}, we obtain
\begin{equation}\label{feq15}
 \mathrm{R}-2(n-1)|\mathrm{div}Z_b|-(n-1)(n-2)|Z_b|^2
 =2(n-1)(\mu-|\rho_b|)
 +(n-1)(n-2)\sum_{\alpha=2}^N|W_\alpha|^2\geqslant0.
\end{equation}
We now claim that 
\begin{equation}\label{feq16}
m_{\mathrm{ADM}}(g)\geqslant|\mathcal{Q}_{Z_b}|=|b^i \mathcal{Q}_{\mathcal{E}_i}|.
\end{equation}
Indeed, notice that replacing $Z_b$ with $-Z_b$ changes the sign of $\mathcal{Q}_{Z_b}$ while leaving assumptions \eqref{feq14} and \eqref{feq15} unchanged; hence, we may assume without loss of generality that $\mathcal{Q}_{Z_b}\geqslant0$.
Moreover, the curvature condition \eqref{feq15} yields $\mathrm{R}_{-(n-1)Z_b}^{(1-n)}\geqslant 0$ and the identity 
$$m_{-(n-1)Z_b}(g)=m_{\mathrm{ADM}}(g)-\mathcal{Q}_{Z_b}$$
holds. Then, the claim follows directly from Theorem \ref{GenPMT} with $X=-(n-1)Z_b$, and it is true for every $b\in\SSS^{N-1}$. Thus, choosing $b^i=\mathcal{Q}_{\mathcal{E}_i}/\mathcal{Q}$ for each $i\in\{1,\dots,N\}$, we conclude 
\begin{equation}\label{feq17}
 m_{\mathrm{ADM}}(g)\geqslant \mathcal{Q}.
\end{equation}
Let us analyze the equality case. If $\mathcal{Q}=0$ and $m_{\mathrm{ADM}}(g)=0$, the classical positive mass theorem \cite{BrendleWang2026, KhuriWangWang} shows that $(M,g)$ is isometric to $(\R^n, g_{\mathrm{eucl}})$, which in turn implies that every $\mathcal{E}_i$ vanishes on $M$ by inequality \eqref{eq:n-dimensional-N-charge-dec-expandedintrobis}. Trivially, if $\mathcal{Q}=0$, $\mathcal{E}_i\equiv0$ for every $i\in\{1,\dots,N\}$, and $(M,g)$ is isometric to $(\R^n, g_{\mathrm{eucl}})$, equality holds in \eqref{feq17}. The rigidity statement $(3)$ follows from rigidity statement $(1)$ by virtue of the observation that $\mathrm{div}\mathcal E_i\equiv0$ yields $\mathcal Q_{\mathcal E_i}=0$, since each $\mathcal{E}_i$ is a smooth vector field on the manifold $M$, which has no boundary.
We now show the rigidity statement $(2)$. Let $\mathcal Q>0$. We suppose that equality holds in \eqref{feq17}. This implies that equality holds in the inequality $m_X\geqslant 0$ of Theorem \ref{GenPMT} with $X=-(n-1)Z_b$, where $b$ is given by 
$$b=\left(\frac{\mathcal Q_{\mathcal E_1}}{\mathcal Q},\dots, \frac{\mathcal Q_{\mathcal E_N}}{\mathcal Q}\right).$$
Therefore, applying the rigidity part of the theorem yields \eqref{gradformofX} and \eqref{eq:critical-rigidity-conformal-flat-nintro}, in particular, $$\mathrm{R}-2(n-1)\mathrm{div}Z_b-(n-1)(n-2)|Z_b|^2\equiv 0.$$
On the other hand, the inequality $\rho_b\,\leqslant\, |\rho_b|\leqslant\,\mu$ and identity \eqref{eq:orthogonal-energy} yield
$$\mathrm{R}-2(n-1)\mathrm{div}Z_b-(n-1)(n-2)|Z_b|^2
 =2(n-1)(\mu-\rho_b)
 +(n-1)(n-2)\sum_{\alpha=2}^N|W_\alpha|^2\geqslant0.$$
Then, we conclude $\mu\equiv \rho_b$ and $W_{\alpha}\equiv0$. 
By equality \eqref{feq18}, the latter implies
$$\mathcal E_i=\frac{\mathcal Q_{\mathcal E_i}}{\mathcal Q} Z_b=\,\frac{\mathcal Q_{\mathcal E_i}}{\mathcal Q}\left(\frac{\mathcal{Q}_{\mathcal{E}_1}}{\mathcal{Q}}\mathcal{E}_1+\dots+\frac{\mathcal{Q}_{\mathcal{E}_N}}{\mathcal{Q}}\mathcal{E}_N  \right).$$
Thus, we obtain
$$\rho_{\mathcal E_i}=\frac{\mathcal Q_{\mathcal E_i}}{\mathcal Q}\rho_b=\frac{\mathcal Q_{\mathcal E_i}}{\mathcal Q}\mu.$$
Conditions \eqref{gradformofX} and \eqref{eq:critical-rigidity-conformal-flat-nintro} state that $X=\nabla f$ and that the conformal Riemannian manifold
$
\left(M,e^{-2f/(n-1)}g\right)
$
is isometric to the Euclidean space. Therefore, the rigidity statement $(2)$ of Theorem \ref{GenPMT} yields $m_X(g)=0$. Since $m_X(g)=m_{\mathrm{ADM}}(g)-\mathcal{Q}$
we conclude that $m_{\mathrm{ADM}}(g)=\mathcal{Q}$.
\end{proof}

The preceding result has a natural interpretation in the three--dimensional Einstein--Maxwell setting. Indeed, on an oriented Riemannian three--manifold, both the electric and magnetic fields can be represented by vector fields, so that the $N$--field framework applies to the pair $(\mathcal E,\mathcal B)$. In this setting, the charged dominant energy condition implies the two--charge energy condition required by the $N$--charge positive mass theorem. We therefore obtain the following charged positive mass theorem.

\begin{proof}[Proof of Theorem \ref{crl:time-symmetric-dyonic-pmt}]
The theorem follows directly from Theorem \ref{thm:Nchargepostivemass} with $\mathcal{E}_1=\mathcal{E}$ and  $\mathcal{E}_2=\mathcal{B}$, since 
$$\mu\geqslant \sqrt{\rho_\mathcal{E}^2+\rho_\mathcal{B}^2+|J_{\mathrm{em}}|^2}\geqslant  \sqrt{\rho_\mathcal{E}^2+\rho_\mathcal{B}^2}.$$
Moreover, whenever equality holds in $(2)$, identities \eqref{linkEeZb} and \eqref{linkBeZb} yield $J_{\mathrm{em}}=0$.
\end{proof}

\section{Higher--dimensional magnetic asymptotics and tensorial mass–charge inequalities}\label{SectHighmahetic}

We start this section by recalling the part of the Einstein--Maxwell initial--data formalism that
motivates the Maxwell--form theorems. Let $(\mathbf{\mathcal M}^{n+1},\mathbf g)$ be an oriented and time--oriented Lorentzian manifold of signature $(-,+,\ldots,+)$.
We say that $(\mathbf{\mathcal M}^{n+1},\mathbf g)$ is an {\em Einstein--Maxwell spacetime} if it satisfies 
$$\mathrm{\mathbf{Ric}}-\frac{1}{2} \mathrm{\mathbf{R}}\mathbf g=(n-1)(n-2) \mathbf T_F+(n-1)\vert \SSS^{n-1}\vert\widehat{\mathbf T}.$$
Here, $\mathbf T_F$ is the electromagnetic energy--momentum tensor coupled with a Maxwell two--form $F\in\Omega^2(\mathcal M)$ satisfying $dF=0$ and $d\star_{\mathbf g}F=0$, which is given by
$$\mathbf T_F(U,V)= \mathbf g(\iota_UF,\iota_VF)-\frac12\|F\|_{\mathbf g}^2\,\mathbf g(U,V)$$
for all smooth vector fields $U,V$ on $\mathbf{\mathcal M}$, while $\widehat{\mathbf T}$ denotes the stress--energy tensor of the non-electromagnetic matter.
Notice that, when $n=3$, the Einstein--Maxwell equation becomes $$\mathrm{\mathbf{Ric}}-(1/2)\mathrm{\mathbf{R}}\mathbf g=2 \mathbf T_F+8\pi\widehat{\mathbf T}.$$
We now consider an oriented spacelike hypersurface $M^n\subset\mathcal M$ with induced metric $g$ and future--directed unit normal $\mathbf N$, and define along $M$
$$e=\iota_M^*\big(  \iota_{\mathbf N}F\big) \quad\quad \text{and}\quad\quad \beta=\iota_M^*(F).$$
We then set $ \mathcal E=e^\sharp\in \Gamma(TM)$ and $ b=\star\beta\in\Omega^{n-2}(M)$.
For every $X,Y\in\Gamma(TM)$, we have
\begin{align}
\bigl(-\mathbf N^\flat\wedge e+\beta\bigr)(X,Y)&=\beta(X,Y)=F(X,Y),\\
\bigl(-\mathbf N^\flat\wedge e+\beta\bigr)(\mathbf N,X)&=-\left(\mathbf N^\flat(\mathbf N)e(X)-N^\flat(X)e(\mathbf N)\right)=e(X)=F(\mathbf N,X),
\end{align}
as $\mathbf N^\flat(\mathbf N)=-1$ and $\mathbf N^\flat(X)=0$. It then follows that
$$
  F=-\mathbf N^\flat\wedge e+\beta.
$$
Moreover, a direct computation gives
$$
  \mathbf T_F(\mathbf N,\mathbf N)
  =
  \frac12
  \bigl(
    |\mathcal E|_g^2+\|b\|_g^2
  \bigr).
$$
With the initial--data convention that the electromagnetic momentum density one--form $ J_{\mathrm{em}}$ is $\iota_M^*\big(\iota_{\mathbf N}\mathbf T_F\big)$, one obtains
$$
  J_{\mathrm{em}} =
  -\iota_{\mathcal E}\beta
  =
  (-1)^{n+1}\star_g(e\wedge b),
$$
by Proposition \ref{prop:hodge-contraction-identity}.
In particular, $|J_{\mathrm{em}}|^2=\|e\wedge b\|^2$ and $b\equiv 0$ implies $J_{\mathrm{em}}\equiv0$. 
In order to introduce the corresponding constraint equations and their time--symmetric version, let $k$ be the second fundamental form of $M$, and define the non--electromagnetic matter energy density and momentum density respectively by
$$
  \widehat\mu
  =
  \widehat{\mathbf T}(\mathbf N,\mathbf N),
  \qquad
  \widehat J
  =\iota_M^*\big(\iota_{\mathbf N}\widehat{\mathbf T}\big).
$$
Then, the Gauss--Codazzi equations yield
\begin{align}
  2(n-1)\vert \SSS^{n-1}\vert\widehat\mu
  &=
  \mathrm{R}
  +(\mathrm{tr}k)^2
  -|k|^2
  -(n-1)(n-2)
   \bigl(
     |\mathcal E|^2+\|b\|_g^2
   \bigr),\\
  (n-1)\vert \SSS^{n-1}\vert\widehat J
  &=
  \mathrm{div}
  \bigl(
    k-(\mathrm{tr}k)g
  \bigr)
  -(n-1)(n-2)J_{\mathrm{em}}.
\end{align}
Consequently, in the time--symmetric case $k=0$, the dominant energy condition $\widehat\mu\geqslant \vert \widehat J\vert$ becomes $$\mu\geqslant (n-2)\vert J_{\mathrm{em}}\vert,$$ where we set $\mu=\vert \SSS^{n-1}\vert\widehat\mu$.
On the other hand, the equations
$d\star_{\mathbf g}F=0$ and $dF=0$ imply respectively
$$
  \mathrm{div}\mathcal E=0
  \qquad \text{and}\qquad
  d\beta=0.
$$
The latter is equivalent to $\delta b=0$.

Theorem \ref{thm:canonical-Maxwell-PMTintro} is deliberately formulated in a more general time--symmetric setting: we impose $k=0$ but we do not impose the source--free constraints. Thus, $\mathrm{div}\mathcal E$ may be nonzero, and the two--form $\beta$ is not assumed to be closed. Furthermore, motivated by the charged dominant energy condition in the purely electric setting \cite{Rau25}, we incorporate the electric source density by imposing the strengthened condition
$$\mu\geqslant\sqrt{\rho_{\mathcal E}^2+(n-2)^2\,|J_{\mathrm{em}}|^2}.$$

\begin{proof}[Proof of Theorem \ref{thm:canonical-Maxwell-PMTintro}]
We start by observing that inequality \eqref{feq19} implies the $1$--charge dominant energy condition
$$ \mathrm{R}-2(n-1)|\mathrm{div} \mathcal{E}|-(n-1)(n-2)|\mathcal{E}|^2\geqslant 0\quad\text{on }M.$$
Therefore, inequality \eqref{feq20} directly follows from Theorem \ref{thm:Nchargepostivemass} with $N=1$ and $\mathcal{E}_1=\mathcal{E}$.\\
If $\mathcal{Q}_{\mathcal E}=0$ and $m_{\mathrm{ADM}}(g)=0$, the same argument used in the proof of the rigidity statement $(1)$ in Theorem \ref{thm:Nchargepostivemass} shows that $(M,g)$ is isometric to $(\R^n, g_{\mathrm{eucl}})$ and $\mathcal{E}$ vanishes on $M$. Then, combining inequality \eqref{feq19} written as
\begin{equation}\label{feq24}
\mathrm{R}\geqslant (n-1)(n-2)
\bigl(|\mathcal E|^2+\|b\|^2\bigr) +2(n-1) \sqrt{\rho_{\mathcal E}^2+(n-2)^2\,|J_{\mathrm{em}}|^2}
  \quad\text{on }M,
  \end{equation}
with $\mathrm{R}\equiv0$, we conclude $b\equiv0$.\\
Trivially, if $\mathcal{Q}_{\mathcal E}=0$, $\mathcal{E}\equiv0$, $b\equiv 0$ and $(M,g)$ is isometric to $(\R^n, g_{\mathrm{eucl}})$, equality holds in \eqref{feq20}.\\ Now, we assume that $\mathcal{Q}_{\mathcal E}\neq 0$ and $m_{\mathrm{ADM}}(g)=|\mathcal{Q}_{\mathcal E}|$. The rigidity statement $(2)$ of Theorem \ref{thm:Nchargepostivemass} yields the claims \eqref{feq21} and \eqref{feq22}, and it also implies 
\begin{equation}\label{feq25}
\mathrm{R}-(n-1)(n-2)|\mathcal{E}|^2=\Big\vert\,\mathrm{R}-(n-1)(n-2)|\mathcal{E}|^2\,\Big\vert=2(n-1)|\mathrm{div} \mathcal{E}|\quad\text{on }M,
\end{equation}
where the first equality is a consequence of inequality \eqref{feq24}. Thus,
$$\mathrm{R}-2(n-1)|\mathrm{div} \mathcal{E}|-(n-1)(n-2)|\mathcal{E}|^2=0\quad\text{on }M,$$
and combining it with inequality \eqref{feq24}, it turns out that $b\equiv 0$. The latter, together with equality \eqref{feq25}, in turn implies that $\mu=|\rho_{\mathcal{E}}|$ on $M$.\\
If there exists a smooth function $f\in C^{2,\alpha}_{-\delta}(M)$ such that \eqref{feq21} and \eqref{feq22} hold, the rigidity statement $(2)$ of Theorem \ref{thm:Nchargepostivemass} implies $m_{\mathrm{ADM}}(g)=|\mathcal{Q}_{\mathcal E}|$.
\end{proof}

The above theorem shows that the magnetic $(n-2)$--form contributes as a nonnegative term to the energy condition, while the resulting mass bound involves only the electric scalar charge. We show that, under suitable integrability assumptions, this is not merely a limitation of the charge--selection argument.

\begin{proof}[Proof of Proposition \ref{prop:tensorial-magnetic-chargeintro}]
We fix an asymptotically flat chart $(U_{\infty}, (x^1,\dots,x^n))$ of order $\tau>(n-2)/2$. 
For any $\omega=\sum_{1\leqslant i_1<\dots<i_{n-3}\leqslant n}a_{i_1\dots i_{n-3}}\,dx^{i_1}\wedge \dots \wedge dx^{i_{n-3}}$ with constant component functions, we have $d\omega=0$, therefore, if we denote by $A_{r,R}$ the annular region $\{r<\vert x\vert< R\}$, we obtain by Stokes' theorem that
\begin{equation}
   \int\limits_{\{\vert x\vert=R\}}\!\!\!\iota_{S_R}^*(\omega\wedge\beta)-\!\!\!\int\limits_{\{\vert x\vert=r\}}\!\!\!\iota_{S_r}^*(\omega\wedge \beta)
   =\int\limits_{A_{r,R}}d(\omega \wedge \beta)
   =(-1)^{n-3}\int\limits_{A_{r,R}}\omega\wedge d\beta,\quad\quad
   \label{eq:magnetic-flux-difference}
\end{equation}
where we are considering on every coordinate hypersphere $\{\vert x\vert=r_0\}$ the induced orientation by $M\setminus\{\vert x\vert\geqslant r_0\}$. Hence, 
\begin{equation}
   \bigg|\, \int\limits_{\{\vert x\vert=R\}}\!\!\!\iota_{S_R}^*(\omega\wedge\beta)-\!\!\!\int\limits_{\{\vert x\vert=r\}}\!\!\!\iota_{S_r}^*(\omega\wedge \beta) \,\bigg|
   \leqslant \binom{n}{3}\,|\omega|_{\mathrm{eucl}}
      \int\limits_{A_{r,R}}|d\beta|_{\mathrm{eucl}}\,d\mu_{\mathrm{eucl}}
      \leqslant C\,|\omega|_{\mathrm{eucl}}
      \int\limits_{A_{r,R}}|d\beta|\,d\mu,
   \label{eq:magnetic-flux-Cauchy}
\end{equation}
for some positive constant $C$ independent of $r$ and $R$, whose existence is guaranteed by the fact that $g_{ij}-\delta_{ij}=O(\vert x \vert ^{-\tau})$ for every $i,j\in \{1,\dots,n\}$.
Now, assumption \eqref{eq:beta-coulomb-decay and dbeta-L1} shows both that the function 
$$r\in [r_0,+\infty)\to\!\!\! \int\limits_{\{\vert x\vert=r\}}\!\!\!\iota_{S_r}^*(\omega\wedge\beta)\in \R$$
is bounded and that the right--hand side of \eqref{eq:magnetic-flux-Cauchy} tends to zero
as $r,\,R\to+\infty$ with $r<R$. Therefore, the limit 
$$\lim\limits_{r\to +\infty}\int\limits_{\{\vert x\vert=r\}}\!\!\!\iota_{S_r}^*(\omega\wedge\beta)$$
exists, is finite and is linear in $\omega$. In particular, it is well defined
$$\frac{1}{|\SSS^{n-1}|}\sum_{1\leqslant i_1<\dots<i_{n-3}\leqslant n}\!\!\!\bigg(\lim\limits_{r\to +\infty}\int\limits_{\{\vert x\vert=r\}}\!\!\!\iota_{S_r}^*\big(dx^{i_1}\wedge \dots \wedge dx^{i_{n-3}}\wedge\beta\big)\bigg)     \,e^{i_1}\wedge \dots \wedge e^{i_{n-3}} \in \Lambda^{n-3}(\mathbb R^n)^*.$$
We now claim that it is equal to $\mathbb Q_b$, given by formula \eqref{eq:directional-magnetic-charge}. Indeed, $\star b=(-1)^{2(n-2)}\beta=\beta$ and, for every $I=(i_1, \dots,i_{n-3})$ with $1\leqslant i_1<\dots<i_{n-3}\leqslant n$, we apply \eqref{eq:hypersurface-hodge-contraction-coordinates} to $\eta=b$, thereby obtaining
\begin{align}
 \iota_{S_r}^*\big(dx^{i_1}\wedge \dots \wedge dx^{i_{n-3}}\wedge\beta\big)&=\iota_{S_r}^*\big(dx^{i_1}\wedge \dots \wedge dx^{i_{n-3}}\wedge\star b\big)\\
 &=(-1)^{n-3}g^{i_1j_1}\cdots g^{i_kj_k} b_{j_1\ldots j_k\ell}\nu^\ell d\sigma\\
 &=(-1)^{n-1}g^{i_1j_1}\cdots g^{i_kj_k} b_{j_1\ldots j_k\ell}\nu^\ell d\sigma.
\end{align}
We then use the asymptotic expansions \eqref{feq26}, \eqref{feq27}, \eqref{feq28}, and \eqref{feq29} by replacing $\delta$ with $\tau$.
This concludes the proof of the claim.\\
We next prove the stronger conclusion \eqref{eq:QB-L1-vanishingintro}. We suppose by contradiction that there exists $I=(i_1, \dots,i_{n-3})$ with $1\leqslant i_1<\dots<i_{n-3}\leqslant n$ such that $(\mathbb Q_b)_{i_1\dots i_{n-3}}\neq 0$. We observe that
\begin{align*} 
\iota_{S_r}^*\big(dx^{i_1}\wedge \dots \wedge dx^{i_{n-3}}\wedge\beta\big)&=\iota_{S_r}^*\Big[d\Big(x^{i_1}dx^{i_{2}}\wedge \dots \wedge dx^{i_{n-3}}\wedge\beta\Big)-(-1)^{n-4}x^{i_1}dx^{i_{2}}\wedge \dots \wedge dx^{i_{n-3}}\wedge d\beta\Big]\\
&=d\Big(\iota_{S_r}^*\big(x^{i_1}dx^{i_{2}}\wedge \dots \wedge dx^{i_{n-3}}\wedge\beta\big)\Big)+(-1)^{n-1}x^{i_1} \iota_{S_r}^*\big(dx^{i_{2}}\wedge \dots \wedge dx^{i_{n-3}}\wedge d\beta\big)
\end{align*}
which by Stokes' theorem implies that
\begin{equation}\label{eq:flux-source-slice-bound}
\bigg\vert\int\limits_{\{\vert x\vert=r\}}\!\!\!\iota_{S_r}^*\big(dx^{i_1}\wedge \dots \wedge dx^{i_{n-3}}\wedge\beta\big)\bigg\vert  \leqslant  \binom{n-1}{3}r\!\!\!\int\limits_{\{\vert x\vert=r\}}\!\!\! |d\beta|_{\mathrm{eucl}}\,d\sigma_{\mathrm{eucl}}=cr\!\!\!\int\limits_{\{\vert x\vert=r\}}\!\!\! |d\beta|_{\mathrm{eucl}}\,d\sigma_{\mathrm{eucl}}
\end{equation}
Therefore, for every sufficiently large $r$, there holds
\begin{equation}
\int\limits_{\{\vert x\vert=r\}}\!\!\! |d\beta|_{\mathrm{eucl}}\,d\sigma_{\mathrm{eucl}}\, \geqslant |\SSS^{n-1}|\,| (\mathbb Q_b)_{i_1\dots i_{n-3}}|/ \left(2cr\right).
\end{equation}
Given that $g_{ij}-\delta_{ij}=O(\vert x \vert ^{-\tau})$ for every $i,j\in \{1,\dots,n\}$ implies that $d\mu=\left[1+O(\vert x\vert^{-\tau})\right]\,d\mu_{\text{eucl}}$, the coarea formula then yields
\begin{equation}
   \int_{\{r\geqslant r_0\}}|d\beta|\,d\mu
   \geqslant C\int_{r_0}^{+\infty}\frac{dr}{r}=+\infty,
\end{equation}
for every sufficiently large $r_0$ and some $C>0$. This is impossible by \eqref{eq:beta-coulomb-decay and dbeta-L1}, and thus \eqref{eq:QB-L1-vanishingintro} holds.
\end{proof}

By Proposition \ref{prop:tensorial-magnetic-chargeintro}, we then associate a tensor--valued charge to a magnetic $(n-2)$--form, provided that it satisfies a prescribed critical asymptotic expansion. This charge is well defined in a suitable sense, and a mass–charge inequality involving its Euclidean norm is established in the following theorem.

\begin{proof}[Proof of Theorem \ref{thm:selected-tensorial-mass-charge}]
We now prove statement $(a)$.
Let $I=(i_1, \dots,i_{n-3})$ with $1~\leqslant~ i_1~<\dots<~i_{n-3}~\leqslant~ n$. We have
$$b_{Ij}\nu_{\mathrm{eucl}}^j=\frac{n}{3}r^{1-n}(dr\wedge\mathbb Q_b)_{Ij}\nu^j_{\mathrm{eucl}}+o(r^{1-n}),$$
and 
$$(dr\wedge\mathbb Q_b)_{Ij}=
\sum_{k=1}^{n-3}
(-1)^{k-1}\nu^{i_k}_{\mathrm{eucl}}
(\mathbb Q_b)_{i_1\cdots\widehat{i_k}\cdots i_{n-3}j}
+
(-1)^{n-3}\nu^{j}_{\mathrm{eucl}}
(\mathbb Q_b)_I,
$$
where $\nu^{i}_{\mathrm{eucl}}=x^i/|x|$ for every $i\in \{1,\dots,n\}$ and the hat denotes omission of the corresponding index. Using the equality
$$
\frac{1}{r^{n-1}|S^{n-1}|}
\int\limits_{\{\vert x \vert\,=\,r\}}\!\!\!\nu^i_{\mathrm{eucl}}\,\nu^j_{\mathrm{eucl}}\,
d\sigma_{\mathrm{eucl}}
=
\frac{\delta_{ij}}{n},
$$
the fact that all the component functions $(\mathbb Q_b)_{j_1\dots j_{n-3}}$ are constant and 
\begin{align}
(\mathbb Q_b)_{i_1\cdots\widehat{i_k}\cdots i_{n-3}i_k}=(-1)^{n-3-k}(\mathbb Q_b)_I,
\end{align}
we obtain as follows:
\begin{align*}
&\frac{(-1)^{n-1}}{|\SSS^{n-1}|} \lim\limits_{r\to +\infty}\int\limits_{\{\vert x \vert\,=\,r\}}\!\!\! b_{Ij}\nu^j_{\mathrm{eucl}}\,d\sigma_{\mathrm{eucl}}\\
&=
\frac{(-1)^{n-1}}{|\SSS^{n-1}|} \lim\limits_{r\to +\infty}\bigg(\,\frac{n}{3}\,r^{1-n}\!\!\!\!\!\int\limits_{\{\vert x \vert\,=\,r\}}\!\!\! (dr\wedge\mathbb Q_b)_{Ij}\nu^j_{\mathrm{eucl}}\,
d\sigma_{\mathrm{eucl}}+o(1)\bigg)\\
&=\lim\limits_{r\to +\infty}\bigg[\frac{(-1)^{n-1}\,n}{3|\SSS^{n-1}|}\,r^{1-n} \!\!\!\!\!\int\limits_{\{\vert x \vert\,=\,r\}}\!\!\!\!\bigg(\sum_{k=1}^{n-3}
(-1)^{k-1}\nu^{i_k}_{\mathrm{eucl}}\nu^j_{\mathrm{eucl}}
(\mathbb Q_b)_{i_1\cdots\widehat{i_k}\cdots i_{n-3}j}
+
(-1)^{n-3}\nu^{j}_{\mathrm{eucl}}\nu^j_{\mathrm{eucl}}(\mathbb Q_b)_I\bigg)\,d\sigma_{\mathrm{eucl}}\bigg]\\
&=\lim\limits_{r\to +\infty}\bigg[\frac{(-1)^{n-1}\,n}{3|\SSS^{n-1}|}\,r^{1-n} \bigg(\sum_{k=1}^{n-3}\frac{ (-1)^{k-1}\,|\SSS^{n-1}|}{n}\, r^{n-1}(\mathbb Q_b)_{i_1\cdots\widehat{i_k}\cdots i_{n-3}i_{_k}}
\!+(-1)^{n-3}r^{n-1}\,|\SSS^{n-1}|(\mathbb Q_b)_I \bigg)\,\bigg]\\
&=\frac{(-1)^{n-1}\,n}{3}\, \bigg[\frac{(-1)^n(n-3)}{n}+(-1)^{n-3}\bigg](\mathbb Q_b)_I
=(\mathbb Q_b)_I.
\end{align*}
We next prove statement $(b)$. Let $\varphi'=(\mathsf{x}^1,\dots, \mathsf{x}^n))$ be another positively oriented asymptotically flat chart of order
$\tau'>(n-2)/2$, and set $\mathsf{r}=\vert \mathsf{x}\vert$. By \cite[Theorem 9.3 and Theorem 9.5]{LeeParker}, the transition maps between these two
charts have the form 
\begin{align}
&\mathsf{x}(x)=\varphi'\circ\varphi^{-1}(x)=Ax+c+O_2\big(r^{1-\min\{\tau,\tau'\}}\big)\\
&x(\mathsf{x})=\varphi\circ (\varphi')^{-1}(\mathsf{x})=A'\mathsf{x}+c'+O_2\big(\mathsf{r}^{1-\min\{\tau,\tau'\}}\big)
\end{align}
with $A,\,A'\in SO(n)$ and for some $c,\,c'\in\mathbb R^n$. Since $\min\{\tau,\tau'\}>(n-2)/2\geqslant 1$, it then follows 
$$\mathsf{r}(x)=|\mathsf{x}(x)|= r+O_1(1)\quad \quad\text{and}\quad \quad r(\mathsf{x})=|x(\mathsf{x})|=\mathsf{r}+O_1(1).$$
Since the transition maps $\varphi'\circ\varphi^{-1}$ and $\varphi\circ (\varphi')^{-1}$ are inverses of one another, we find that $AA'=I$ and $Ac'+c=0$. Moreover, for every $i\in\{1,\dots,n\}$, we have
$$(\varphi'\circ\varphi^{-1})_{*} (dx^i)
=\frac{\partial (\varphi^i\circ (\varphi')^{-1})}{\partial \mathsf{x}^{j}}\,d\mathsf{x}^{j}=\Big((A^{-1})^i_{j}+ O_1\big(\mathsf{r}^{-\min\{\tau,\tau'\}}\big)  \Big)\,d\mathsf{x}^{j}=A_{i}^j\,d\mathsf{x}^{j}+O_1\big(\mathsf{r}^{-\min\{\tau,\tau'\}}\big)
$$
where we used the orthogonality of $A$ and the convention that $A^{i}_j$ denotes the element in the $i$--th row and the $j$--th column of $A$. Therefore, for every $I=(i_1, \dots,i_{n-3})$ with $1\leqslant i_1<\dots<~i_{n-3}\leqslant ~n$, we obtain
\begin{align}
(\varphi'\circ\varphi^{-1})_{*} (dx^{i_1}\wedge \dots \wedge dx^{i_{n-3}})&=A_{i_1}^{j_1} \dots A_{i_{n-3}}^{j_{n-3}}\,d\mathsf{x}^{j_1} \wedge \dots \wedge  d\mathsf{x}^{j_{n-3}}+O_1\big(\mathsf{r}^{-\min\{\tau,\tau'\}}\big)\\
&=\!\!\!\!\!\sum_{\substack{ J=(j_1, \dots,j_{n-3})\\1\leqslant j_1<\dots<j_{n-3}\leqslant n}}\!\!\!\!\det(A_{I}^J)\,d\mathsf{x}^{j_1} \wedge \dots \wedge  d\mathsf{x}^{j_{n-3}}+O_1\big(\mathsf{r}^{-\min\{\tau,\tau'\}}\big).
\end{align}
Recall that by our convention, $A_{I}^{J}$ denotes the $(n-3)\times (n-3)$ submatrix of $A$ with rows indexed by $J$ and columns indexed by $I$. 
In particular, there holds 
\begin{align*}
(\varphi'\circ\varphi^{-1})_{*} (\mathbb{Q}_b)&=\!\!\!\!\sum_{\substack{I=(i_1, \dots,i_{n-3})\\1\leqslant i_1<\dots<i_{n-3}\leqslant n}}\,\,\sum_{\substack{J=(j_1, \dots,j_{n-3})\\1\leqslant j_1<\dots<j_{n-3}\leqslant n}}\!\!\!\! (\mathbb{Q}_b)_{I}\det(A_{I}^J)\,d\mathsf{x}^{j_1} \wedge \dots \wedge  d\mathsf{x}^{j_{n-3}}+O_1\big(\mathsf{r}^{-\min\{\tau,\tau'\}}\big)\\
&=\!\!\!\!\sum_{\substack{J=(j_1, \dots,j_{n-3})\\1\leqslant j_1<\dots<j_{n-3}\leqslant n}}\!\Bigg(\sum_{\substack{I=(i_1, \dots,i_{n-3})\\1\leqslant i_1<\dots<i_{n-3}\leqslant n}} \!\!\!\! (\mathbb{Q}_b)_{I}\det(A_{I}^J) \Bigg)\,d\mathsf{x}^{j_1} \wedge \dots \wedge  d\mathsf{x}^{j_{n-3}}+O_1\big(\mathsf{r}^{-\min\{\tau,\tau'\}}\big)\\
&=\mathbb{Q}'_b+O_1\big(\mathsf{r}^{-\min\{\tau,\tau'\}}).
\end{align*}
Combining the previous expansions and using that $A$ is orthogonal yield
$$(\varphi'\circ\varphi^{-1})_{*} (r^{1-n}dr\wedge\mathbb{Q}_b)=\mathsf{r}^{1-n} d\mathsf{r}\wedge \mathbb{Q}'_b+O_1\big(\mathsf{r}^{-n}),$$
as $\min\{\tau,\tau'\}>(n-2)/2\geqslant 1$. Therefore, by $ (\varphi'\circ\varphi^{-1})_{*}\big(o_1(r^{1-n})\big)=o_1(\mathsf{r}^{1-n})$, we obtain the claimed transformation law. Finally, by the Cauchy--Binet formula, we have
\begin{align*}
\|\mathbb Q'_b\|^2_{\mathrm{eucl}}&=\!\!\!\!\sum_{\substack{J=(j_1, \dots,j_{n-3})\\1\leqslant j_1<\dots<j_{n-3}\leqslant n}}\!\!\Bigg(\sum_{\substack{I=(i_1, \dots,i_{n-3})\\1\leqslant i_1<\dots<i_{n-3}\leqslant n}} \!\!\!\! (\mathbb{Q}_b)_{I}\det(A_{I}^J)\Bigg)^2\\
&=\!\!\!\!\sum_{\substack{J=(j_1, \dots,j_{n-3})\\1\leqslant j_1<\dots<j_{n-3}\leqslant n}}\!\!\Bigg(\sum_{\substack{I=(i_1, \dots,i_{n-3})\\1\leqslant i_1<\dots<i_{n-3}\leqslant n}} \!\!\!\! (\mathbb{Q}_b)_{I}\det(A_{I}^J)\Bigg)\!\!\Bigg(\sum_{\substack{K=(\kappa_1, \dots,\kappa_{n-3})\\1\leqslant \kappa_1<\dots<\kappa_{n-3}\leqslant n}} \!\!\!\! (\mathbb{Q}_b)_{K}\det(A_{K}^J)\Bigg)\\
&=\!\!\!\! \sum_{\substack{I=(i_1, \dots,i_{n-3})\\1\leqslant i_1<\dots<i_{n-3}\leqslant n}}\,\,\sum_{\substack{K=(\kappa_1, \dots,\kappa_{n-3})\\1\leqslant \kappa_1<\dots<\kappa_{n-3}\leqslant n}} \!\!\!\!( \mathbb{Q}_b)_{I}(\mathbb{Q}_b)_{K} \, \Bigg(\sum_{\substack{J=(j_1, \dots,j_{n-3})\\1\leqslant j_1<\dots<j_{n-3}\leqslant n}}\det(A_{I}^J)\det(A_{K}^J)\Bigg)\\
&=\!\!\!\! \sum_{\substack{I=(i_1, \dots,i_{n-3})\\1\leqslant i_1<\dots<i_{n-3}\leqslant n}}\,\,\sum_{\substack{K=(\kappa_1, \dots,\kappa_{n-3})\\1\leqslant \kappa_1<\dots<\kappa_{n-3}\leqslant n}} \!\!\!\!( \mathbb{Q}_b)_{I}(\mathbb{Q}_b)_{K}\, \det\!\big[(A^{-1})^I A_K\big]=\!\!\!\!\!\!\! \sum_{\substack{I=(i_1, \dots,i_{n-3})\\1\leqslant i_1<\dots<i_{n-3}\leqslant n}}\! \!\!\!\!\!\!\big((\mathbb{Q}_b)_{I}\big)^2\!=\|\mathbb Q_b\|^2_{\mathrm{eucl}}
\end{align*}
where $(A^{-1})^I$ is the $(n-3)\times n$ submatrix of $A^{-1}$ with rows indexed by $I$ and $A_K$ is the $n\times (n-3)$ submatrix of $A$ with columns indexed by $K$.\\
We now show statement $(c)$. Let 
$$\eta=\!\!\! \sum_{1\leqslant i_1<\dots<i_{n-3}\leqslant n}\!\!\! (\mathcal Q_b)^{-1}\,(\mathbb Q_b)_{i_1\dots i_{n-3}}\,x^{i_1}dx^{i_2}\wedge \dots \wedge dx^{i_{n-3}},$$
and we notice that $d\eta= (\mathcal Q_b)^{-1}\,\mathbb Q_b$.
We choose a smooth cut--off function $\chi$ that vanishes on a sufficiently
large compact subset of $M$ and is identically one near infinity, and define 
$$\omega=d(\chi \eta).$$
By construction, $\omega$ is a smooth and closed $(n-3)$--form on $M$ and is equal to $(\mathcal Q_b)^{-1}\,\mathbb Q_b$ near infinity. 
Using $\omega$, we define the smooth vector field $Z$ by the identity $Z^\flat=(-1)^{n-1}\star(\omega\wedge\beta).$ 
Notice that, near infinity, we have
\begin{align*}
Z^\flat&=(-1)^{n-1}\star(\omega\wedge\beta)=\frac{(-1)^{n-1}}{\mathcal Q_b}\!\!\!\!\!\sum_{1\leqslant i_1<\dots<i_{n-3}\leqslant n}\!\!\!(\mathbb Q_b)_{i_1\dots i_{n-3}}\,\iota_{(dx^{i_{n-3}})^{\sharp}}\dots\iota_{(dx^{i_{1}})^\sharp}b\\
&=\frac{(-1)^{n-1}}{\mathcal Q_b}\!\!\!\!\!\sum_{1\leq i_1<\dots<i_{n-3}\leq n} (\mathbb Q_b)_{i_1\dots i_{n-3}} g^{i_1 j_1}\dots g^{i_{n-3} j_{n-3}}\, b_{j_1\dots j_{n-3} j}dx^j\\
&=\frac{(-1)^{n-1}}{\mathcal Q_b}\!\!\!\!\!\!\!\sum_{\substack{I=(i_1, \dots,i_{n-3})\\1\leqslant i_1<\dots<i_{n-3}\leqslant n}}\!\!\!\!\!(\mathbb Q_b)_{I}\Bigg(\frac{n}{3}r^{1-n} \sum_{k=1}^{n-3}
(-1)^{k-1}\nu^{i_k}_{\mathrm{eucl}}
(\mathbb Q_b)_{i_1\cdots\widehat{i_k}\cdots i_{n-3}j}
+
(-1)^{n-3}\frac{n}{3}r^{1-n}\nu^{j}_{\mathrm{eucl}}
(\mathbb Q_b)_{I}\Bigg) dx^j\\
&\quad+o_{1}(r^{1-n}),
\end{align*}
by formula \eqref{eq:hodge-contraction-general} and by computations analogous to those in the proof of statement $(a)$, as $g_{ij}-\delta_{ij}=O_2(\vert x \vert ^{-\tau})$ for every $i,j\in \{1,\dots,n\}$. In particular, $\mathcal Z=O_1(r^{1-n})$, and hence is admissible.
Now, we check that 
$$\mathcal{Q}_{\mathcal{Z}}=\frac{1}{|\mathbb S^{n-1}|}
  \lim_{r\to\infty}
  \int\limits_{\{|x|=r\}}\!\!\!
  \delta_{ij}\mathcal Z^i \nu^{j}_{\mathrm{eucl}}
  \,d\sigma_{\mathrm{eucl}}=\mathcal{Q}_b.$$
Then, by proceeding analogously to the proof of statement $(a)$, we obtain 
$$\mathcal{Q}_{\mathcal{Z}}=(\mathcal Q_b)^{-1}\!\!\!\!\!\!\!\!\!\sum_{\substack{I=(i_1, \dots,i_{n-3})\\1\leqslant i_1<\dots<i_{n-3}\leqslant n}}\!\!\!\!\!\big((\mathbb Q_b)_{I}\big)^2=\mathcal Q_b.$$
Statement $(d)$, including the equality case, follows from Theorem \ref{GenPMT} at the critical value $k=1-n$, after just one clarification.
By our assumptions, $X$ is admissible, and $\mathrm{R}+2\mathrm{div}X\in L^1(M,g)$, therefore, the $X$--ADM mass $m_X(g)$ is well defined. On the other hand, statement $(c)$ along with the fact that $\mathcal{E}$ is an admissible smooth vector field with $\mathrm{div}(\mathcal{E})\in L^{1}(M,g)$ implies that
$$ \lim_{r\to\infty}
  \int\limits_{\{|x|=r\}}\!\!\!
  \delta_{ij} X^i \nu^{j}_{\mathrm{eucl}}
  \,d\sigma_{\mathrm{eucl}}=-(n-1)\vert \SSS^{n-1}\vert\mathcal{Q}.
$$
As a consequence of this result, together with the fact that $m_X(g)$ given by formula \eqref{formXADMmass} is finite, it follows that $m_{\mathrm{ADM}}(g)$ is well defined (precisely, the limit in the distinguished asymptotically flat chart exists and is finite, while its invariance under changes of positively oriented asymptotically flat charts follows from that of $m_X(g)$ together with that of the charges introduced),
hence,
\begin{equation}
m_X(g)
=
m_{\mathrm{ADM}}(g)-\mathcal{Q}.
\end{equation}
\end{proof}

\appendix

\section{Hodge star--interior multiplication identities}

We collect here some useful Hodge star--interior multiplication identities.

\begin{proposition}
\label{prop:hodge-contraction-identity}
Let $(M^n,g)$ be an oriented Riemannian manifold and $1\leqslant k\leqslant \ell\leqslant n$. Then, for every $\omega_1,\dots,\omega_k\in \Omega^1(M)$ and $\eta\in\Omega^\ell(M)$, one has
\begin{equation}\label{eq:hodge-contraction-general}
\star\bigl(\omega_1\wedge\cdots\wedge \omega_k\wedge\star\eta\bigr)
=(-1)^{(\ell-k)(n-\ell)}\,
\iota_{\omega_k^{\sharp}}\dots \iota_{\omega_1^{\sharp}}\,\eta.
\end{equation}
\end{proposition}

\begin{proof}
Let $(e_1,\ldots,e_n)$ be a local positively oriented orthonormal frame of $M$, with the dual coframe
$(\vartheta^1,\ldots,\vartheta^n)$. For every $I=(i_1, \dots,i_{k})$ with $1\leqslant i_1<\dots<i_{k}\leqslant n$ and $J=(j_1,\dots,j_\ell)$ with $1\leqslant j_1<\cdots<j_\ell\leqslant n$, we denote by $\vartheta^I$ the $k$--form $\vartheta^{i_1}\wedge\cdots\wedge\vartheta^{i_k}$, and similarly for $\vartheta^J$. We then show that 
\begin{equation}\label{eq:hodge-contraction-generalforbasis}
\star\bigl(\vartheta^I\wedge\star\vartheta^J\bigr)
=(-1)^{(\ell-k)(n-\ell)}
\iota_{e_{i_k}}\dots \iota_{e_{i_1}}\vartheta^J,
\end{equation}
and note that this is sufficient to establish the identity \eqref{eq:hodge-contraction-general}, by multilinearity.
If $I\not\subset J$, then $\iota_{e_{i_k}}\dots \iota_{e_{i_1}} \vartheta^J=0$. On the other hand, we also have $\vartheta^I \wedge \star \vartheta^J= 0$. Indeed, in this case $\star \vartheta^J = \pm \vartheta^{J^c}$ and the wedge product necessarily contains a repeated basis 1--form. 
Assume now that $I\subset J$. Let $K=(\kappa_{k+1}, \dots,\kappa_{l})$ with $1\leqslant \kappa_{k+1}<\dots<\kappa_{l}\leqslant n$ be such that
$\{j_1,\dots,j_\ell\}=\{i_1,\dots,i_{k}, \kappa_{k+1}, \dots,\kappa_{l}\}$. Then, we have
\begin{equation}
\label{eq:contraction-basis}
\iota_{e_{i_k}}\dots \iota_{e_{i_1}}\vartheta^J
=\mathrm{sgn}(\sigma)\,\vartheta^K,
\end{equation}
where $\sigma$ is the permutation of the multi--index sets such that $ \vartheta^J=\mathrm{sgn}(\sigma)\,\vartheta^I\wedge\vartheta^K$.
Let $J^c$ denote the ordered complement of $J$ in $\{1,\ldots,n\}$ and
write
$$
\vartheta^1\wedge\cdots\wedge\vartheta^n=\mathrm{sgn}(\tau)\,
\vartheta^J\wedge\vartheta^{J^c},
$$
where $\tau$ is the corresponding permutation of the multi--index sets.
By the definition of the Hodge star operator,
$$
\star_g\vartheta^J=\mathrm{sgn}(\tau)\,\vartheta^{J^c}.
$$
Therefore,
$$
\star\bigl(\vartheta^I\wedge\star\vartheta^J\bigr)
=\mathrm{sgn}(\tau)
\star\!\bigl(\vartheta^I\wedge\vartheta^{J^c}\bigr).
$$
The complement of $I\cup J^c$ is precisely $K$. To determine the sign, we observe that
\begin{align}
\vartheta^1\wedge\cdots\wedge\vartheta^n&=\mathrm{sgn}(\tau)\,
\vartheta^J\wedge\vartheta^{J^c}=\mathrm{sgn}(\sigma)\, \mathrm{sgn}(\tau)\,
\vartheta^I\wedge \vartheta^K\wedge\vartheta^{J^c}\\
&=(-1)^{(\ell-k)(n-\ell)}\,\mathrm{sgn}(\sigma)\,\mathrm{sgn}(\tau)\,
\vartheta^I\wedge \vartheta^{J^c}\wedge\vartheta^K.
\end{align}
It then follows that
$$
\star\bigl(\vartheta^I\wedge\vartheta^{J^c}\bigr)
=(-1)^{(\ell-k)(n-\ell)}\,\mathrm{sgn}(\sigma)\,\mathrm{sgn}(\tau)\,\vartheta^K.
$$
Consequently, we obtain
\begin{equation}
\star_g\bigl(\vartheta^I\wedge\star_g\vartheta^J\bigr)=
(-1)^{(\ell-k)(n-\ell)}\,\mathrm{sgn}(\sigma)\,\big(\mathrm{sgn}(\tau)\big)^2\,\vartheta^K
=(-1)^{(\ell-k)(n-\ell)}\,\mathrm{sgn}(\sigma)\,\vartheta^K.
\end{equation}
Combining this identity with \eqref{eq:contraction-basis} proves
\eqref{eq:hodge-contraction-generalforbasis}.
\end{proof}

\begin{proposition}
\label{prop:hypersurface-hodge-contraction}
Consider an oriented Riemannian manifold $(M^n,g)$, and suppose $\Sigma^{n-1}\subset M$ is an hypersurface with the orientation determined
by a unit normal vector field $\nu$ (hence, $d\sigma=\iota_{\Sigma}^*(\iota_\nu d\mu)\,$).
Let $1\leqslant k\leqslant n-1$. Then, for every $\eta\in\Omega^{k+1}(M)$ and $\omega_1,\dots,\omega_k\in \Omega^1(M)$, there holds
\begin{equation}
\label{eq:hypersurface-hodge-contraction}
\iota_{\Sigma}^*\bigl(\omega_1\wedge\cdots\wedge \omega_k\wedge\star\eta\bigr) 
=
(-1)^k\bigl(\iota_{\omega_k^{\sharp}}\dots \iota_{\omega_1^{\sharp}}\,\eta\bigr)(\nu)\,d\sigma.
\end{equation}
In particular, in local coordinates $(x^1,\ldots,x^n)$, one has
\begin{equation}
\label{eq:hypersurface-hodge-contraction-coordinates}
\iota_{\Sigma}^*\bigl(dx^{i_1}\wedge\cdots\wedge dx^{i_k}\wedge\star\eta\bigr)
=
(-1)^k
g^{i_1j_1}\cdots g^{i_kj_k}
\eta_{j_1\ldots j_k\ell}\nu^\ell
\,d\sigma
\end{equation}
for every $I=(i_1, \dots,i_{k})$ with $1\leqslant i_1<\dots<i_{k}\leqslant n$.
\end{proposition}

\begin{proof}
We first recall an elementary identity for $(n-1)$--forms. Let
$$
\gamma\in\Omega^{n-1}(M).
$$
We choose a local positively oriented orthonormal frame $(e_1,\ldots,e_{n-1})$ on $\Sigma$. Note that $(\nu,e_1,\ldots,e_{n-1})$ is then a positively oriented orthonormal frame along an open set $S$ of $\Sigma$, and let
$
(\nu^\flat,\vartheta^1,\ldots,\vartheta^{n-1})
$
be its dual coframe. 
By the definition of the induced
orientation, we have
$$
d\sigma
=
\vartheta^1\wedge\cdots\wedge\vartheta^{n-1}\quad\quad \text{and}\quad\quad 
\star\!
\bigl(
\vartheta^1\wedge\cdots\wedge\vartheta^{n-1}
\bigr)
=
(-1)^{n-1}\nu^\flat
$$
on $S$.
It then follows that
$$
(\star\gamma)(\nu)
=
(-1)^{n-1}\,
\gamma(e_1,\ldots,e_{n-1}).
$$
Therefore, we have
\begin{equation}
\label{eq:pullback-n-minus-one-form}
\iota_{\Sigma}^*\gamma
= \gamma(e_1,\ldots,e_{n-1})\,\vartheta^1\wedge\cdots\wedge\vartheta^{n-1}=
(-1)^{n-1}\,
(\star\gamma)(\nu)
\,d\sigma.
\end{equation}
We now apply \eqref{eq:pullback-n-minus-one-form} to $\gamma=\omega_1\wedge\cdots\wedge \omega_k\wedge\star\eta$, which has degree $n-1$, and then use Proposition \ref{prop:hodge-contraction-identity}, thereby obtaining
\begin{align}
\iota_{\Sigma}^*\bigl(\omega_1\wedge\cdots\wedge \omega_k\wedge\star\eta\bigr)
&=
(-1)^{n-1}
\star\!\bigl(\omega_1\wedge\cdots\wedge \omega_k\wedge\star\eta\bigr)(\nu)
\,d\sigma\\
&=
(-1)^{n-1}(-1)^{n-k-1}\bigl(\iota_{e_{i_k}}\dots \iota_{e_{i_1}}\eta\bigr)(\nu)\,d\sigma\\
&=
(-1)^k
\bigl(\iota_{e_{i_k}}\dots \iota_{e_{i_1}}\eta\bigr)(\nu)\,d\sigma.
\end{align}
This proves identity \eqref{eq:hypersurface-hodge-contraction}.\\
For the coordinate expression, we have
\begin{align}
\bigr(\iota_{(dx^{i_k})^{\sharp}}\dots \iota_{(dx^{i_1})^{\sharp}}\,\eta\bigr) (\nu)&=\eta\bigl( (dx^{i_1})^{\sharp},\dots, (dx^{i_k})^{\sharp},\nu \bigr)\\
&=\eta\bigl( g^{i_1j_1}\partial_{j_1},\dots,g^{i_kj_k}\partial_{j_k} ,\nu^l \partial_{l}\bigr)\\
&=g^{i_1j_1}\cdots g^{i_kj_k}\eta_{j_1\ldots j_k\ell}\nu^\ell,
\end{align}
 and substituting this identity into
\eqref{eq:hypersurface-hodge-contraction} gives
\eqref{eq:hypersurface-hodge-contraction-coordinates}.
\end{proof}

\bigskip
\noindent
\textbf{{Acknowledgements.}} {\em  The author deeply thanks Carlo Mantegazza. %for his thorough reading of the manuscript and for the many precious suggestions.
The author also thanks Christian B\"{a}r and Emanuele Spadaro.
The author is member of INdAM–GNAMPA, and is partially supported by the GNAMPA project ``Analysis of Non--smooth Geometric Evolutions''.}\\
\textbf{AI usage statement} {\em ChatGPT 5.6 Pro was used as interactive assistant. The author takes full responsibility for the content of the paper.}

\bigskip

\bibliographystyle{amsplain}
\bibliography{biblio}

\end{document}